\documentclass[letter,12pt]{report}
\usepackage[mathscr]{eucal}
\usepackage{graphicx}    
\usepackage{setspace} 
\usepackage{hyperref}
\usepackage{calc}
\usepackage{fancyhdr}
\usepackage{amsmath}
\usepackage[alphabetic]{amsrefs}
\usepackage{lipsum} 
\usepackage[top=1in, bottom=1in, left=1in, right=1in]{geometry} 
\usepackage[titles]{tocloft}

\usepackage[utf8]{inputenc}
\usepackage{genyoungtabtikz}
\usepackage{ytableau}
\usepackage{amssymb}
\usepackage{amsthm}
\usepackage{cite}
\usepackage{mathtools}
\usepackage{fullpage}
\usepackage{authblk}
\usepackage{array}
\usepackage{caption}
\newcolumntype{P}[1]{>{\centering\arraybackslash}p{#1}}

\theoremstyle{plain}
\newtheorem{theorem}{Theorem}
\newtheorem{corollary}[theorem]{Corollary}
\newtheorem{lemma}[theorem]{Lemma}
\newtheorem{proposition}[theorem]{Proposition}
\theoremstyle{definition}

\newtheorem{example}[theorem]{Example}

\theoremstyle{remark}

\newcommand{\N}{\mathbb{{N}}}
\newcommand{\Z}{\mathbb{Z}}

\newenvironment{rSection}[1]{ 
  \sectionskip
  \MakeUppercase{\bf #1} 
  \sectionlineskip
  \hrule 
  \begin{list}{}{ 
    \setlength{\leftmargin}{1.5em} 
  }
  \item[]
}{
  \end{list}
}

\def\sectionlineskip{\medskip} 
\def\sectionskip{\medskip} 
	
\usepackage{titlesec}

\titleformat{\chapter}{\centering\normalfont\large}{\thechapter}{1em}{} 
\titleformat{\section}{\centering\normalfont\large}{\thesection}{1em}{}
\titleformat{\subsection}{\centering\normalfont\large}{\thesubsection}{1em}{}

		\titlespacing*{\chapter}{0pt}{-0.25in}{20pt} 
		
		\fancypagestyle{sxoli2}{
\fancyhf{}
\pagenumbering{arabic}
\rhead{\thepage}

        \fancypagestyle{plain}{%
        \rhead{\thepage}
        }
}

\usepackage{etoolbox}
\makeatletter
\patchcmd{\chapter}{\if@openright\cleardoublepage\else\clearpage\fi}{}{}{}
\makeatother  

\begin{document}

\setcounter{page}{1}
\thispagestyle{empty}

\begin{center}
{\Large THE GINI INDEX IN ALGEBRAIC COMBINATORICS AND REPRESENTATION THEORY}\\
\vspace{0.6\baselineskip}
by\\
Grant Kopitzke\\
\vspace{1in}
A Dissertation Submitted in\\
Partial Fulfillment of the\\
Requirements for the Degree of\\
\vspace{1in}
Doctor of Philosophy \\
in Mathematics\\ 
\vspace{1in}
at\\
The University of Wisconsin-Milwaukee\\
May 2021\\
\end{center}
		\newpage
		\clearpage

		
		\pagestyle{plain}
\pagenumbering{roman}
\setcounter{page}{2}

		\begin{center}
		\singlespacing
		{\large ABSTRACT} \\
		THE GINI INDEX IN ALGEBRAIC COMBINATORICS AND REPRESENTATION THEORY\\
		\doublespacing
		by\\
		
		Grant Kopitzke\\
		\singlespacing
		The University of Wisconsin-Milwaukee, 2021 \\
		
		Under the Supervision of Dr. Jeb Willenbring\\
		\doublespacing
		\end{center}
		\begin{flushleft}
		\doublespacing
The Gini index is a number that attempts to measure how equitably a resource is distributed throughout a population, and is commonly used in economics as a measurement of inequality of wealth or income. The Gini index is often defined as the area between the ``Lorenz curve'' of a distribution and the line of equality, normalized to be between zero and one. In this fashion, we will define a Gini index on the set of integer partitions and prove some combinatorial results related to it; culminating in the proof of an identity for the expected value of the Gini index. These results comprise the principle contributions of the author, some of which have been published in \cite{Kopitzke} .\\

We will then discuss symmetric polynomials, and show that the Gini index can be understood as the degrees of certain Kostka-foulkes polynomials. This identification yields a generalization whereby we may define a Gini index on the irreducible representations of a complex reflection group, or connected reductive linear algebraic group. \\
		\end{flushleft}
		\newpage
		\topskip0pt
        \vspace*{\fill}
		\begin{center}
		    \textit{For Amanda}
		\end{center}
        \vspace*{\fill}
        \vspace*{\fill}
		\newpage
				
		 \renewcommand*\contentsname{TABLE OF CONTENTS}

		    \tableofcontents
				\newpage
				
				    \renewcommand*\listfigurename{LIST OF FIGURES}
					
    \listoffigures

		\newpage
	
    \listoftables

\newpage
\doublespacing
		    \chapter*{ACKNOWLEDGMENTS}

\thispagestyle{plain}
\pagenumbering{roman}
\setcounter{page}{8}  
				\doublespacing
    There are many people without whom this work would not have been possible. Chiefly among them is my advisor, Dr. Jeb Willenbring. His advice, patience, insight, guidance, and encouragement have surpassed anything I could have expected. I am extremely honored and grateful to have worked with him. By the same token, I thank my dissertation committee members; Professors Allen Bell, Kevin McLeod, Boris Okun, and Yi Ming Zou for their instruction, readership and comments. Nearly all of the mathematics I know is due to the hard work of these professors, and I am forever indebted to them. 
    
    Outside of UW Milwaukee, but within the realm of mathematics, I first wish to thank Dr. Carrie Tirel of UW Oshkosh, Fox Cities; for giving me the opportunity to tutor in the math lab, for convincing me to go on to graduate school, and for sparking within me a love of teaching that has guided my career choices ever since. Also at UW Oshkosh, I would like to thank Dr. David Penniston, my undergraduate thesis advisor, who pushed me harder and further than I thought I could go at the time, and who taught me what it really means to do mathematical research.
    
    Lastly, I wish to thank my family. Even though they may not understand the contents of this work, their support proved invaluable throughout the process. First and foremost, I want to thank my wife Amanda Kopitzke. Throughout this process, the only person who has been more patient with me than my advisor, was Amanda. Her constant love, support, and reassurance have been truly overwhelming. Next my father and mother, Lynn and Barbara Kopitzke, I thank for their generosity, confidence, and support. I thank my brother, Dr. Ryan Ross for his advice and encouragement throughout the job hunting and interviewing process. Finally I thank my sisters, Grace and Amy Lucas, and my brother Tyler Ross. Research can be a lonely process, especially during a quarantine. Despite this, Grace, Amy, and Tyler were always a text away to keep me company. For that I am very grateful. 
    
    Many more people have touched my life along this journey --- far too many to thank individually by name. However, their presence, assistance, and guidance is acknowledged and greatly appreciated.

		\newpage

\clearpage

\pagenumbering{arabic}
	\doublespacing
\chapter{Introduction}
\label{Chapter: Introduction}
There are two primary goals to this work. The first, addressed in Chapter 3, is to investigate the combinatorial properties of the discrete Gini index, $g$, defined on the set, $P_n$, of partitions of a positive integer $n$. We prove a convenient identity relating the Gini index, $g$, to the second elementary symmetric polynomial $e_2$. This identity provides us with a generating function for the Gini index
\[\prod_{n=1}^{\infty}\frac{1}{1-q^{\binom{n+1}{2}}x^n}-1=  \sum_{n=1}^{\infty}\sum_{\lambda\vdash n}q^{\left(\binom{n+1}{2}-g(\lambda) \right)}x^n.\]

From here, we analyze two different properties of the Gini index: its dominance properties (known as ``Schur convexity''), and its expected value on the set $P_n$. We show that the generating function for the Gini index provides us with easily computed lower bounds on the length of the maximum antichain in the dominance lattice - which touches on a longstanding open problem in the theory of integer partitions. Finally, in the end of Chapter 3 we prove an identity by which one can easily calculate the expected value of the Gini index on the set $P_n$. 

The second goal of this work is to frame the discrete Gini index, $g$, within the structure of representation theory. This function occurs naturally as the degrees of certain Kostka-Foulkes polynomials $K_{\lambda, \mu}$, which we discuss in Chapter 4. The first connections to representation theory are seen in Chapter 5, where the Gini index appears as the degrees of the graded multiplicities of irreducible representations of the symmetric group, $S_n$, inside the coinvariant ring of $S_n$. We then extend the notion of the Gini index to one defined on the irreducible representations of a complex reflection group. 

The discrete Gini index $g$ has a natural extension, $g_{nk,n}$, discussed in Section 3.6, which is directly related to the representation theory of the general linear group $GL_n(\mathbb{C})$. In Chapter 6 we show that this ``extended'' Gini index occurs as the degree of the graded multiplicity of an irreducible rational representation of $GL_n(\mathbb{C})$ inside the harmonic polynomials of $GL_n(\mathbb{C})$. As was done in Chapter 5, we extend this notion to define a Gini index on the irreducible representations of a connected reductive linear algebraic group over $\mathbb{C}$.

\newpage
\chapter{Preliminaries}
\label{Chapter: Preliminaries}
Unless otherwise stated, throughout this dissertation the ground field will always be the complex numbers, and vector spaces will always be assumed to be finite dimensional over $\mathbb{C}$. Furthermore, representations are always assumed to be linear, and finite dimensional. If $\rho:G\longrightarrow GL(V)$ is a representation of a group $G$, we will usually suppress either the map $\rho$ or the vector space $V$.

\section{Integer Partitions}
A \textit{partition}, $\lambda$, of a positive integer $n$ (sometimes written as ${\lambda \vdash n}$) is a sequence ${(\lambda_1 , \lambda_2 ,\ldots, \lambda_{\ell})}$ of ${\ell\leq n}$ decreasing non-negative integers such that ${\sum_{i=1}^{\ell}{\lambda_{i}}=n}$. The ${\lambda_i}$ (${1\leq i \leq \ell}$) are called the ``parts'' of $\lambda$. To avoid repeating parts, it is sometimes useful to write a partition as $(\lambda_1^{a_1},\lambda_2^{a_2},\ldots,\lambda_{\ell}^{a_{\ell}})$ to represent $\lambda_i$ repeating $a_i$ times. In this case we have that $\sum_{i=1}^{\ell}{{a_i}\lambda_i}=n$, and $\lambda_i\neq \lambda_j$ for all $i \neq j$. This notation will be used in the proof of Proposition \ref{prop3}. In order to make the length of $\lambda$ (the number of parts) equal to $n$, one can ``pad out'' the partition by adding ${n-\ell}$ zeros to the end. 
\begin{example}
The partition ${(4,3,1,1)}$ of 9 is equivalent to ${(4,3,1,1,0,0,0,0,0).}$ This identification will be used when defining the Lorenz curve of a partition.
\end{example}
A \textit{Young diagram} is a finite collection of boxes arranged in left-justified rows, with a weakly decreasing number of boxes in each row (see \cite{Fulton}). Integer partitions are in one to one correspondence with Young diagrams in the following way: if ${\lambda = (\lambda_1 , \lambda_2 ,\ldots, \lambda_{\ell})}$ is a partition of $n$, then the Young diagram of shape ${\lambda}$ has ${\lambda_1}$ boxes in its first row, ${\lambda_2}$ boxes in its second row, etc. 
\begin{example}
If ${\lambda = (4,3,1,1)}$, then the Young diagram of shape $\lambda$ is
\[\yng(4,3,1,1)\,.\]
\end{example}
The \textit{conjugate partition} $\widetilde{\lambda}$ of $\lambda$ is the partition of $n$ obtained by reflecting the Young diagram of $\lambda$ across its main diagonal. As in the previous example, if ${\lambda = (4,3,1,1)}$, then the Young Diagram of $\widetilde{\lambda}$ is 
\[\yng(4,2,2,1)\,,\]
hence ${\widetilde{\lambda}=(4,2,2,1)}$. Conjugation is clearly a bijection on the partitions of $n$.

The \textit{dominance order} is a partial order on the set of partitions of $n$.  If $\lambda=(\lambda_1,\lambda_2,\ldots,\lambda_n)$ and $\mu=(\mu_1,\mu_2,\ldots,\mu_n)$ are partitions of $n$, then $\mu \preceq \lambda$ if 
\[\sum_{i=1}^{k}{\mu_i} \leq \sum_{i=1}^{k}{\lambda_i} \]
for all $k\geq 1$. It is well known that conjugation of partitions is an antiautomorphism on the dominance lattice of partitions of $n$ (see \cite{Brylawski}). In other words, if $\mu \preceq \lambda$, then $\widetilde{\lambda}\preceq\widetilde{\mu}$. We will write $\mu \prec \lambda$ if $\mu \preceq \lambda$ and $\mu \neq \lambda$, and will let $P_n$ denote the partially ordered set of partitions of $n$ with respect to dominance.

Let $\lambda$ be a partition of $n$. A \emph{tableau of shape} $\lambda$ is a filling of the Young diagram of $\lambda$ with numbers from $[n]=\{1,\ldots,n\}$. A \emph{semistandard tableau} or \emph{column strict tableau} is a filling by positive integers in $[n]=\{1,2,\ldots,n\}$ that is
\begin{enumerate}
    \item weakly increasing across each row, and
    \item strictly increasing down each column.
\end{enumerate}
For brevity, if $\lambda$ is a partition of $n$, then we will often use $\lambda$ to refer to the partition or the Young diagram of shape $\lambda$ interchangeably.
A \emph{standard tableau} of shape $\lambda$ is a semistandard tableau in which each number in $[n]$ occurs exactly once.
\begin{example}
If $\lambda=(3,2,1,1)\vdash 7$, then
\[\young(122,23,3,7)\]
is a semistandard tableau, whereas
\[\young(125,34,6,7)\]
is a standard tableau.
\end{example}

\section{Irreducible representations of the symmetric group}
If $G$ is any finite group, then a finite dimensional \emph{representation} of $G$ over $\mathbb{C}$ is a group homomorphism $G\longrightarrow GL(V)$, where $V$ is a finite dimensional complex vector space, and $GL(V)$ is the group of invertible linear transformations on $V$. The contents of chapter 5 deal with representations of finite complex reflection groups. The canonical example of such a group is the symmetric group, $S_n$, of permutations of $[n]=\{1,\ldots,n\}$. In this section we will cover the classification of irreducible representations of $S_n$ over $\mathbb{C}$.

The number of irreducible representations of $S_n$, like any finite group, is the number of conjugacy classes. As the conjugacy classes of $S_n$ are indexed by cycle types, their number is $P(n)$; the number of partitions of $n$. There is, in fact, a natural indexing of the irreducible representations of $S_n$ by these partitions via a construction called the \emph{Specht module} of a partition. Our construction of the Specht modules follows that in \cite{Fulton}.

Let $\lambda$ be a partition of $n$. The symmetric group, $S_n$, acts on the set of all numberings of the Young diagram of $\lambda$ with numbers from $[n]$ --- each of which occur exactly once. If $T$ is such a numbering of $\lambda$, and $\sigma\in S_n$, then the action of $\sigma$ on $T$, $\sigma\cdot T$, yields the numbering of $\lambda$ which has the number $\sigma(i)$ in the same box in which $i$ occurs in $T$. 
\begin{example}
Let $\sigma=(12345)\in S_5$, and let \[ T=\young(123,45).  \]
Then \[ \sigma\cdot T=\young(512,34). \]
\end{example}
For a numbering $T$ of $\lambda$, the \emph{row group}, $R(T)$, and \emph{column group}, $C(T)$, of $T$ are the subgroups of $S_n$ defined by
\begin{align*}
    R(T)&=\{\sigma\in S_n:\sigma\text{ permutes the entries of each row of $T$ among themselves}\},\text{ and}\\
    C(T)&=\{\sigma\in S_n:\sigma\text{ permutes the entries of each column of $T$ among themselves}\}.
\end{align*}
\begin{example}
When $S_5$ acts on the numbering
\[ T=\young(123,45) \]
of $\lambda=(3,2)$, the row and column groups of $T$ are
\begin{align*}
    R(T)&=\{\sigma\tau:\sigma\in S_3\text{ and }\tau\in\{1,(45)\}  \},\text{ and} \\
    C(T)&=\{1,(14),(25),(14)(25)  \}.
\end{align*}
\end{example}
A \emph{tabloid} is an equivalence class of numberings of a Young diagram $\lambda\vdash n$ with distinct numbers from $[n]$. Under this relation, two numberings of $\lambda$ are considered to be equivalent if their row groups are the same. We denote by $\{T\}$ a tabloid containing the numbering $T$. The symmetric group acts on the set of tabloids by
\[ \sigma\cdot\{T\}=\{\sigma\cdot T\} .\]

Let $\mathbb{C}[S_n]$ denote the group ring of $S_n$, which consists of all complex linear combinations of permutations of $[n]$, where multiplication is determined by composition in $S_n$. A representation of $S_n$ is the same as a left $\mathbb{C}[S_n]$-module. Given a numbering $T$ of the Young diagram of shape $\lambda$, we define the \emph{Young symmetrizers} of $T$ as the elements
\begin{align*}
    a_T&=\sum_{p\in R(T)}p,\\
    b_T&=\sum_{q\in C(T)}\text{sgn}(q)q,\text{ and }\\
    c_T&=b_T\cdot a_T.
\end{align*}
\begin{example}
The Young symmetrizers $a_T$ and $b_T$ of the filling
\[T=\young(123,45)  \]
are
\begin{align*}
    a_T&=\left[1+(12)+(13)+(23)+(123)+(132)  \right]\left[1+(45) \right] ,\text{ and}\\
    b_T&= 1-(14)-(25)+(14)(25).
\end{align*}
\end{example}
Define $M^{\lambda}$ to be the complex vector space with basis the tabloids $\{T\}$ of shape $\lambda$. Since $S_n$ acts on the set of tabloids of shape $\lambda$, it acts on $M^{\lambda}$ - which makes $M^{\lambda}$ a left $\mathbb{C}[S_n]$-module. For each numbering $T$ of $\lambda$, there is a special element of $M^{\lambda}$ defined by the formula
\[ v_{T}=b_T\cdot\{T\}. \]
We then define the \emph{Specht module} $S^{\lambda}$ to be the subspace of $M^{\lambda}$ spanned by the elements $v_{T}$ as $T$ varies over all numberings of $\lambda$. The Specht module $S^{\lambda}$ is then a $\mathbb{C}[S_n]$-submodule of $M^{\lambda}$. Putting this all together, we obtain the following theorem.

\begin{theorem}{Classification of Irreducible Representations of $S_n$ (\cite{Fulton} ) }\\
For each partition $\lambda$ of $n$, $S^{\lambda}$ is an irreducible representation of $S_n$. Every irreducible representation of $S_n$ is isomorphic to exactly one $S^{\lambda}$. 
\end{theorem}

\begin{example}
\begin{enumerate}
    \item For any $n\in\mathbb{N}$, the symmetric group $S_n$ has a one-dimensional representation called the \emph{trivial representation}, defined by
\[ \sigma\cdot x=x, \]
for all $\sigma\in S_n$ and $x\in\mathbb{C}$. If $\lambda=(n)$, then $S^{(n)}$ is the trivial representation of $S_n$.

\item If $n\geq 2$, then $S_n$ has a one-dimensional representation called the \emph{sign representation}, defined by
\[\sigma\cdot x=\text{sgn}(\sigma)x,  \]
for all $\sigma\in S_n$ and $x\in\mathbb{C}$. Here, the \emph{sign} of $\sigma$ is defined to be $1$ if $\sigma$ is a product of an even number of transpositions, and is $-1$ otherwise. If $\lambda=(1^n)$, then $S^{(1^n)}$ is the alternating representation of $S_n$. 

\item if $n> 2$, then $S_n$ has a $(n-1)$-dimensional irreducible representation called the \emph{standard representation}. Choose a basis $\{e_1,\ldots,e_n\}$ of $\mathbb{C}^n$, and define an action of $S_n$ on $\mathbb{C}^n$ by
\[\sigma\cdot(a_1e_1+\cdots+a_ne_n)=a_1e_{\sigma(1)}+\cdots+a_ne_{\sigma(n)},  \]
where $\sigma\in S_n$ and $\alpha_1,\ldots,\alpha_n\in\mathbb{C}$. This action defines an $n$-dimensional reducible representation of $S_n$ called the \emph{permutation representation}. It is reducible because it has a 1-dimensional subspace spanned by
\[ e_1+\cdots+e_n. \]
The orthogonal compliment of this 1-dimensional space is the $(n-1)$-dimensional irreducible subspace of $\mathbb{C}^n$ spanned by the vectors $(v_1,\ldots,v_n)\in\mathbb{C}^n$ such that $v_1+\cdots+v_n=0$. This is called the \emph{standard representation} of $S_n$. If $\lambda=(n-1,1)$, then $S^{(n-1,1)}$ is the standard representation of $S_n$.
\item Let $V=S^{(n-1,1)}$ be the standard representation of $S_n$. If $n>3$ and $0\leq k\leq n-1$ then $S_n$ has a $\binom{n-1}{k}$-dimensional irreducible representation, the \emph{$k^{th}$ exterior power} of the standard representation, denoted by $\Lambda^k V$. If $\lambda=(n-k,1^k)$, then $S^{(n-k,1^k)}$ is isomorphic to $\Lambda^k V$.

\end{enumerate}
\end{example}

\section{Linear algebraic groups}
\label{Linear Algebraic Groups}
The contents of chapter 6 deal with representations of certain linear algebraic groups. Before discussing the main topics of this dissertation, we will specify the exact objects and morphisms in the category to be considered in that chapter. For precision, the groups of chapter 6 are always \emph{linear algebraic groups} over the complex numbers. Such groups are defined as affine varieties with a compatible group structure, and are always isomorphic to a Zariski-closed subgroup of the complex general linear group, $GL_n(\mathbb{C})$, of $n\times n$ invertible complex matrices (for some $n\in\mathbb{N})$. If $G\subseteq GL_n(\mathbb{C})$ is a algebraic group, then a \emph{rational representation} of $G$ is a group homomorphism, $G\longrightarrow GL_m(\mathbb{C})$ for some $m\in\mathbb{N}$, which is a regular morphism in the category of affine varieties. 

To each linear algebraic group $G$ there is an associated Lie algebra $\mathfrak{g}$, which is defined as the derivations (infinitesimal transformations) of the regular functions $\mathcal{O}[G]$ that commute with left translations (see \cite{SRI}). By $\mathcal{O}[G]$ we mean the algebra of regular (rational) functions on $G$, which are defined as the restriction to $G$ of the regular functions
\[ \mathcal{O}[GL_n(\mathbb{C})]=\mathbb{C}[x_{11},x_{12},\ldots,x_{nn},\det(x)^{-1}] \]
on $GL_n(\mathbb{C})$, where $x=[x_{ij}]\in GL_n(\mathbb{C})$. The Lie algebra of a linear algebraic group has a natural embedding into the $n\times n$ complex matrices $M_n(\mathbb{C})$, and is equipped with a Lie bracket 
\[[\,\bullet\,,\bullet\,]:\mathfrak{g}\times\mathfrak{g}\longrightarrow\mathfrak{g},  \]
which can be defined in terms of matrices as
\[ [X,Y]=XY-YX, \]
for $X,Y\in \mathfrak{g}$. Henceforth we will always view a Lie algebra as being comprised of matrices. Any element $X\in\mathfrak{g}$ yields a linear transformation
\[\text{ad}_X(Y)=[X,Y],  \]
which is called the \emph{adjoint representation} of $\mathfrak{g}$. The representation theories of $G$ and its Lie algebra $\mathfrak{g}$ are closely connected. In fact, a representation of $G$ is irreducible if and only if its differential is an irreducible representation of the Lie algebra $\mathfrak{g}$ (c.f. \cite{SRI} Theorem 2.2.7). Hence $\mathfrak{g}$ plays a significant role in the classification of the irreducible representations of $G$.

We will further restrict the linear algebraic groups under consideration to those that are connected and reductive. A linear algebraic group $G$ is \emph{connected} if $\mathcal{O}[G]$ has no zero divisors. $G$ is \emph{reductive} if every rational representation of $G$ is completely reducible (hence the name ``reductive''). In other words, every rational representation of a reductive group can be written as a sum of irreducible representations. The complete reducibility of rational representations makes the study of reductive groups and the classification of their irreducible representations particularly nice.

A toral subgroup $T\subseteq G$ is a subgroup that is isomorphic to $(\mathbb{C}^{\times})^m$, for some non-negative integer $m$. A torus is \emph{maximal} if it is not properly contained within any other torus. If the group $G$ is connected and reductive, then all maximal tori are conjugate and have the same dimension (known as the \emph{rank} of $G$). Furthermore, every element of $G$ lies in a maximal torus. In this case, the lie algebra ,$\mathfrak{h}$, of a maximual torus, $T$, is called a \emph{Cartan} subalgebra of $\mathfrak{g}$. Cartan subalgebras are maximal abelian subalgebras of $\mathfrak{g}$ in which every element is semisimple. Just as with maximal tori, all Cartan subalgebras are conjugate in $\mathfrak{g}$.

\section{The Theorem of the Highest Weight}
\label{The Theorem of the Highest Weight}

Let $G$ be a connected reductive linear algebraic group of rank $n$, and fix a choice $T\cong (\mathbb{C}^{\times})^n$ of maximal torus in $G$. Let $\mathfrak{g}$ and $\mathfrak{h}$ be the Lie algebras of $G$ and $T$, respectively. Let $\mathfrak{h}^*$ denote the algebraic dual space of $\mathfrak{h}$; that is, the set of all linear functionals $\mathfrak{h}\longrightarrow\mathbb{C}$.

Since $T$ is abelian, its irreducible representations are all 1-dimensional. Denote by $\mathscr{X}(T)$ the group (under tensor products of representations) of all irreducible representations of $T$. The group $\mathscr{X}(T)$ is isomorphic to $\mathbb{Z}^n$. The \emph{weight lattice} of $G$ is the set 
\[ P(G)=\{d\theta:\theta\in\mathscr{X}(T)\}\subseteq\mathfrak{h}^* \]
of differentials of irreducible representations of $T$ (i.e., the irreducible representations of $\mathfrak{h}$ corresponding to those of $T$). The elements of $P(G)$ are called \emph{weights} of $G$. If $\alpha\in P(G)$ is a weight of $G$, we define
\[ \mathfrak{g}_{\alpha}=\{X\in\mathfrak{g}:[A,X]=\alpha(A)X\text{ for all }A\in\mathfrak{h}\}. \]
If $\alpha=0$, then $\mathfrak{g}_0=\mathfrak{h}$. If $\alpha$ is nonzero and $\mathfrak{g}_{\alpha}$ is nonzero, then we call $\alpha$ a \emph{root} of $T$ on $\mathfrak{g}$, and we call $\mathfrak{g}_{\alpha}$ a root space. We call the set $\Phi\subseteq\mathfrak{h}^*$ of roots the \emph{root system} of $\mathfrak{g}$, with respect to our choice $T$ of maximal torus. The root system $\Phi$ spans $\mathfrak{h}^*$. 

A subset 
\[ \Delta=\{\alpha_1,\ldots,\alpha_n\}\subseteq \Phi \]
is a set of \emph{simple roots} of $\mathfrak{g}$ if every $\gamma\in\Phi$ can be written uniquely as
\[ \gamma=n_1\alpha_1+\cdots+\cdots+n_k\alpha_k, \]
with $n_1,\ldots,n_k$ integers all of the same sign. The simple roots $\Delta$ form a basis for $\mathfrak{h}^*$, and partition the root system $\Phi$ into two disjoint subsets
\[\Phi=\Phi^+\cup\Phi^-,\]
where $\Phi^+$ consists of all roots for which the coefficients $n_1,\ldots,n_k$ are non-negative. We call $\gamma\in\Phi^+$ a \emph{positive root} relative to $\Delta$.

For a root $\alpha\in\Phi$, we call an element $h_{\alpha}\in[\mathfrak{g}_{\alpha},\mathfrak{g}_{\alpha}]$ such that $\alpha(h_{\alpha})=2$ a \emph{coroot} of $\alpha$. The \emph{weight lattice} of $\mathfrak{g}$ is defined as
\[ P(\mathfrak{g})=\{\lambda\in\mathfrak{h}^*:\lambda(h_{\alpha})\in\mathbb{Z}\text{ for all }\alpha\in\Phi\}, \]
and the \emph{dominant integral weights} of $\mathfrak{g}$ (relative to the choice of simple roots $\Delta$) are defined as,
\[P_{++}(\mathfrak{g})=\{\lambda\in\mathfrak{h}^*:\lambda(h_{\alpha})\in\mathbb{Z}_{\geq 0}\text{ for all }\alpha\in\Phi \}.  \]
Finally, the \emph{dominant weights} of $G$ are defined to be the weights of $G$ that are also dominant integral weights of the Lie algebra $\mathfrak{g}$: 
\[P_{++}(G)=P(G)\cap P_{++}(\mathfrak{g}).  \]

There is a partial order defined on the set of weights of $\mathfrak{g}$ (and thus on the dominant weights of $G$). Let $\lambda,\mu\in \mathfrak{h}^*$ be two dominant weights of $\mathfrak{g}$. Since $\Phi^+$ spans $\mathfrak{h^*}$, $\lambda$ and $\mu$ can be written as linear combinations of the positive roots. If $\lambda-\mu$ can be written as a linear combination of the positive roots with non-negative real coefficients, then we say that $\lambda$ is \emph{higher} than $\mu$. This partial order then restricts to the dominant weights of $G$. Let $V$ be a (finite dimensional) irreducible representation of $G$. A weight $\lambda$ of $V$ is called the \emph{highest weight} of $V$ if $\lambda$ is higher than every other weight $\mu$ of $V$.

The irreducible representations of $G$ are classified by a result known as the ``theorem of the highest weight''.

\begin{theorem}{Theorem of the Highest Weight}
If $G$ is a connected reductive linear algebraic group and $T$ is a maximal torus in $G$, the following results hold
\begin{enumerate}
    \item Every irreducible representation of $G$ has a highest weight.
    \item Two irreducible representations of $G$ with the same highest weight are isomorphic.
    \item The highest weight of each irreducible representation of $G$ is dominant.
    \item If $\lambda$ is a dominant weight of $G$, there exists an irreducible representation of $G$ with highest weight $\lambda$.
\end{enumerate}
\end{theorem}

\section{The General Linear Group}
Fix a positive integer $n$. The general linear group $GL_n(\mathbb{C})$ of complex $n\times n$ invertible matrices is a connected reductive linear algebraic group. As a affine variety, it is defined as the set of points $(x,y)\in \mathbb{C}^{n^2}\times \mathbb{C}$ satisfying the polynomial equation
\[ y\det(x)=1. \]
The canonical choice of maximal torus in $GL_n(\mathbb{C})$ is the subgroup $T$ of diagonal matrices. The Lie algebra of $GL_n(\mathbb{C})$ is the set
$\mathfrak{gl}_n(\mathbb{C})=M_n(\mathbb{C})$
of all $n\times n$ complex matrices, and the Lie algebra of $T$ is the subalgebra
\[ \mathfrak{h}=\{X\in\mathfrak{gl}_n(\mathbb{C}):X\text{ is diagonal}\}.\]
Define $\varepsilon_i\in\mathfrak{h}^*$ such that \[ \varepsilon_i(A)=a_i, \]
for any $A=\text{diag}[a_1,\ldots,a_n]\in\mathfrak{h}$. The root system of $T$ is the set
\[ \Phi=\{\varepsilon_i-\varepsilon_j:1\leq i\neq j\leq n\} .\]
The standard choice for the set of simple roots is
\[\Delta=\{\varepsilon_i-\varepsilon_{i+1}:1\leq i <n\}, \]
and the corresponding set of positive roots is
\[\Phi^+=\{\varepsilon_i-\varepsilon_j:1\leq i<j\leq n  \}  .\]
The weight lattice $P(GL_n(\mathbb{C}))$ can then be written in terms of these weights as
\[P(GL_n(\mathbb{C}))=\bigoplus_{k=1}^n\mathbb{Z}\varepsilon_k.  \]
The set of dominant integral weights of $\mathfrak{g}$ is given by
\[ P_{++}(\mathfrak{g})=\{k_1\varepsilon_1+\cdots+k_n\varepsilon_n:k_1\geq k_2\geq\cdots\geq k_n\text{ and }k_i-k_{i+1}\in\mathbb{Z} \}. \]
Hence the dominant weights of $GL_n(\mathbb{C})$ are
\begin{align*}
    P_{++}(GL_n(\mathbb{C}))&=P(GL_n(\mathbb{C}))\cap P_{++}(\mathfrak{g})\\
    &=\{k_1\varepsilon_1+\cdots+k_n\varepsilon_n:k_1\geq k_2\geq\cdots\geq k_n\text{ and }k_i\in\mathbb{Z}\}.
\end{align*}
By the theorem of the highest weight, if $V$ is an irreducible representation of $GL_n(\mathbb{C})$, then $V$ has a highest weight, and its highest weight is dominant. If $\lambda\in P_{++}(GL_n(\mathbb{C}))$ is the highest weight of $V$, then we write $V=V^{\lambda}$ and call $V^{\lambda}$ a ``highest weight representation'' of $\lambda$. \\

For a more detailed look at the information provided in this chapter, we refer the reader to \cite{SRI}.

\newpage
\chapter{The Gini Index}

\section{The ``Standard'' Gini Index}
Much of this first section has appeared in the \textit{Journal of Integer Sequences} in a paper by the author (see \cite{Kopitzke}).\\

In part one of his 1912 book ``Variabilit\`a e Mutabilit\`a" (Variability and Mutability), the statistician Corrado Gini formulated a number of different summary statistics; among which was what is now known as the \textit{Gini index} - a measure that attempts to quantify how equitably a resource is distributed throughout a population. Referring to ``the" Gini index can be misleading, as no fewer than thirteen formulations of his famous index appeared in the original publication \cite{Origins}. Since then, many others have appeared in a variety of different fields.

The Gini index is usually defined using a construction known as a \textit{Lorenz Curve}. In ``Methods of Measuring the Concentration of Wealth", Lorenz defined this curve in the following fashion. Consider a population of people amongst whom is distributed some fixed amount of wealth.
Let ${L(x)}$ be the percentage of total wealth possessed by the poorest $x$ percent of the population. The graph ${y=L(x)}$ is the Lorenz curve of the population \cite{Lorenz}.

It is clear from this definition that $L(0)=0$ (I.E., none of the people have none of the wealth), ${L(1)=1}$ (all of the people have all of the wealth), and $L$ is non-decreasing. Since any population of people must have finite size $n$, the function ${L(x)}$ as defined above would appear to be a discrete function on the set ${\{\frac{k}{n}:k\in \mathbb{Z}\text{ and }0\leq k \leq n\}}$. However, in practice $L$ is often made continuous on the interval ${[0,1]}$ by linear interpolation \cite{Farris}.

If each person possesses the same amount of wealth, then the Lorenz curve for this distribution is the line $y=x$, which we call the ``line of equality". The area between the line of equality and the Lorenz curve of a wealth distribution provides a measurement of the wealth inequality in that population. \begin{figure}[htp]
    \centering
    \includegraphics[width=10cm]{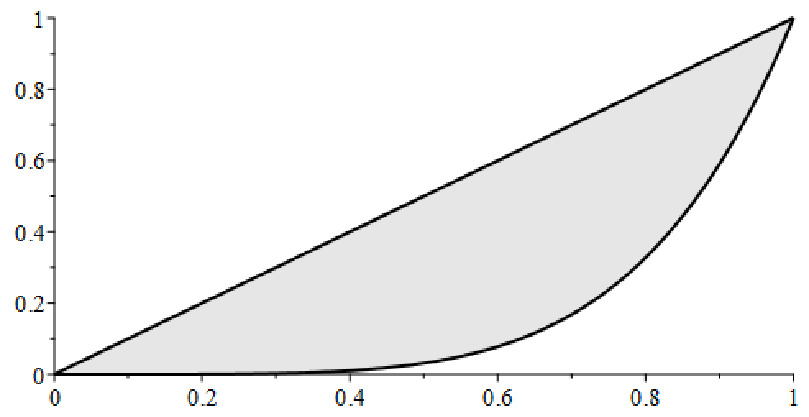}
    \caption{Area between the line of equality and a typical Lorenz curve}
\end{figure}

The maximum possible area of ${\frac{1}{2}}$ arises from the distribution in which one person controls all of the wealth (${L(1)=1}$, and ${L(x)=0}$ for all ${x\neq1}$). The Gini index of a distribution is then defined by calculating the area between the line of equality and Lorenz curve of the distribution, and normalizing this area to be between zero and one:
\[G=2\int_{0}^{1}{\big(x-L(x)\big)dx}.\]

In this paper we consider distributions of a discrete indivisible resource in a finite population, where the amount of that resource is equal to the number of people in the population. There is a natural one-to-one correspondence between the set of such distributions with $n$ people, and the set of partitions of $n$. We will then define the Gini index of a partition in a similar fashion as above.

\section{The second elementary symmetric polynomial $e_2$}
The \textit{second elementary symmetric polynomial}, $e_2$, in $n$ variables, ${x_1,\,x_2,\ldots\,x_n}$, is defined \[e_2(x_1,x_2,\ldots,x_n)=\sum_{1\leq i < j \leq n}{x_ix_j}.\] 
\begin{example}
If ${\lambda=(4,3,1,1)}$ is a partition of $9$, then 
\[e_2(\lambda)=\Big(4(3+1+1)+3(1+1)+1(1) \Big)=27.\]
\end{example}
We will make use of the following result.
\begin{lemma}
If ${\lambda=(\lambda_1, \lambda_2 ,\ldots, \lambda_{n})}$ is a partition of a positive integer $n$, then 
\[e_2(\lambda)=\binom{n}{2}-\sum_{i=1}^{n}\binom{\lambda_i}{2}.\]

\label{lemma1}
\end{lemma}
\begin{proof}
Let ${\lambda=(\lambda_1,\lambda_2,\ldots,\lambda_{n})}$ be a partition of $n$. Note that ${\sum_{i=1}^{n}{\lambda_i}=n}$. Then 
\begin{align*}
    e_2(\lambda) &= \sum_{1\leq i < j \leq n}{\lambda_i \lambda_j}\\
    &=\binom{n}{2}-\left( \sum_{1\leq i < j \leq n}{\big(-\lambda_i \lambda_j \big)}+\binom{n}{2}\right)\\
    &=\binom{n}{2}-\frac{1}{2}\left( \sum_{1\leq i < j \leq n}{\big(-2\lambda_i \lambda_j \big)}+n(n-1)\right)\\
    &=\binom{n}{2}-\frac{1}{2}\left( \sum_{1\leq i < j \leq n}{\big(-2\lambda_i \lambda_j \big)}+\left(\sum_{i=1}^{l}{\lambda_i}\right)\left(\sum_{j=1}^{l}{\lambda_j}-1\right)\right)\\
    &=\binom{n}{2}-\frac{1}{2}\left( \sum_{1\leq i < j \leq n}{\big(-2\lambda_i \lambda_j \big)}+\sum_{i=1}^{l}{\left(\lambda_{i}^{2}-\lambda_i\right)} + \sum_{1\leq i<j\leq l}{\big(2\lambda_i \lambda_j\big)} \right)\\
    &=\binom{n}{2}-\frac{1}{2}\left( \sum_{i=1}^{n}{\lambda_{i}(\lambda_i-1)} \right)\\
    &=\binom{n}{2}-\left( \sum_{i=1}^{n}{\binom{\lambda_i}{2}} \right).\\
\end{align*}
\end{proof}

\section{The Gini Index of an Integer Partition}
As previously stated, we restrict our study of the Gini index to finite populations where the amount of a discrete indivisible resource is equal to the size of the population. In other words, there is one of said resource available for each person. The distributions of $n$ of such a resource amongst $n$ people is in one-to-one correspondence with the integer partitions of $n$.

\begin{example}
If there are $4$ dollars in a population of $4$ people, then the partition ${(3,1)}$ of $4$ would correspond to one person having $3$ dollars, one person having $1$ dollar, and the two remaining people having nothing. Whereas the partitions ${(1,1,1,1)}$ and ${(4)}$ correspond to completely equitable and completely inequitable distributions, respectively.
\end{example}

Given a partition $\lambda=(\lambda_1,\ldots,\lambda_n)$ of a positive integer $n$ (padded with zeros on the tail, if necessary), the $\textit{Lorenz curve of $\lambda$}$, $L_{\lambda}:[0,n]\longrightarrow [0,n]$, is defined as $L_{\lambda}(0)=0$, and $L_{\lambda}(x)=\sum_{i=n-k+1}^{n}{\lambda_i}$, where $1\leq k\leq n$ is the unique positive integer such that $x\in(k-1,k]$.
In other words, for $k$ from $1$ to $n$, the Lorenz curve of $\lambda$ on the interval $(k-1,k]$ is the sum of the last $k$ parts of $\lambda$, $\lambda_n+\lambda_{n-1}+\dots+\lambda_{n-k+1}$. Since total equality corresponds to the flat partition $(1^n)$, using the above definition for the Lorenz curve of a partition, we find that the line of equality is given by $y=\lceil x \rceil$.

\begin{figure}[htp]
    \centering
    \includegraphics[width=10cm]{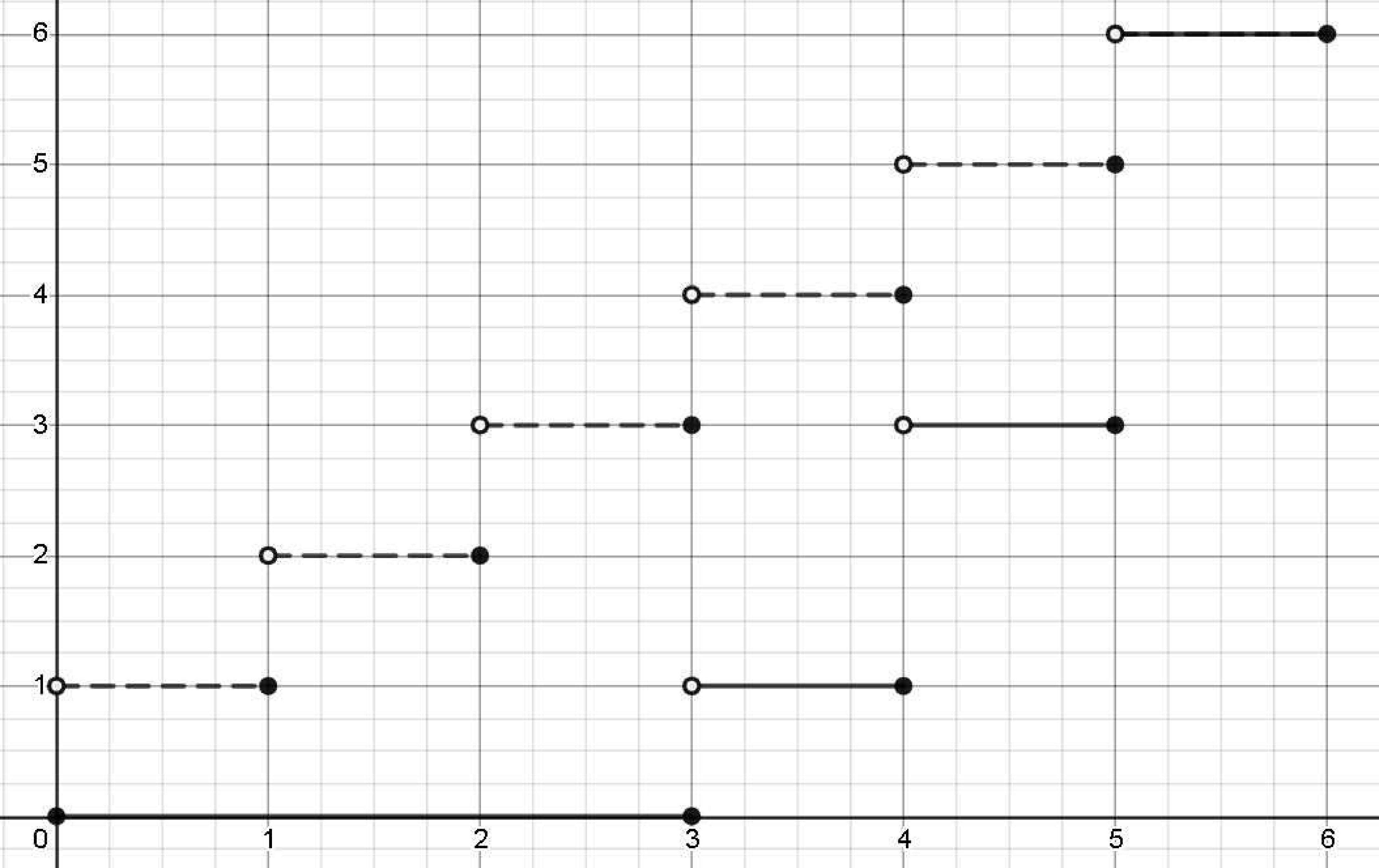}
    \caption{The line of equity (dashed) and the Lorenz curve of the partition (3,2,1) of 6 (solid).}
\end{figure}

The standard Gini index is calculated by finding the area between the line of equality and the Lorenz curve, and normalizing. In a similar fashion we define the Gini index, $g$, of a partition $\lambda=(\lambda_1,\ldots,\lambda_n)$ of $n$ by 
\begin{align*}
    g(\lambda)&= \int_{0}^{n} \left( \lceil x \rceil -L_{\lambda}(x) \right) dx\\
    &=\binom{n+1}{2}-\sum_{i=1}^{n}{i\lambda_i}\\
    &=\binom{n}{2}-\sum_{i=1}^{n}{(i-1)\lambda_i}.
\end{align*}

The ordinary Gini index is normalized to be between zero and one. For a fixed value of $n$, the function $g$ attains its maximum value of $\binom{n}{2}$ on the partition $(n)$ of $n$. So the Gini index of a partition $\lambda$ of $n$ can be normalized by dividing $g(\lambda)$ by $\binom{n}{2}$. As long as $n$, and ${g(\lambda)}$ are both known, the normalized Gini index of $\lambda$ can always be calculated in this fashion. With this in mind, we may disregard the normalization, and view $g$ itself as the integer valued ``discrete'' Gini index of a partition.

The sums
\[ \sum_{i=1}^{n}\binom{\widetilde{\lambda_i}}{2} \]
and
\[ \sum_{i=1}^{n}(i-1)\lambda_i \]
have appeared in our formulas for $e_2$ and $g$, respectively. It is known (cf. \cite{Greene}) that these two quantities are equal. This fact, in conjunction with Lemma \ref{lemma1} yield some interesting results.

\begin{proposition} 
If $\lambda$ is an integer partition, then ${g(\lambda)=e_{2}(\widetilde{\lambda})}$, where ${\widetilde{\lambda}}$ is the conjugate partition of $\lambda$.
\label{prop2}
\end{proposition}

In light of Lemma 11, this result follows from statements in \cite{Greene}. These statements are, however, given without proof, so we will now provide proofs of these facts.

\begin{proof}
Let ${\lambda=(\lambda_1,\lambda_2,\ldots,\lambda_{n})}$ be a partition of a positive integer $n$, where ${\lambda_1\geq\lambda_2\geq\ldots\geq\lambda_{n}>0}$ and ${\sum_{i=1}^{n}{\lambda_i}=n}$. We can calculate $g{(\lambda)}$ by filling the Young diagram of shape $\lambda$ with numbers, where the entry in any box counts the number of boxes in that column that are strictly above it. For example, for the partition ${(4,3,1,1)}$, we would have
\[\young(0000,111,2,3)\,.\]
Then the sums of the values in each row are
\begin{align*}
    \sum{\big(\textrm{Entries in row }1\big)}&=0\lambda_1,\\
    \sum{\big(\textrm{Entries in row }2\big)}&=1\lambda_2,\\
    \sum{\big(\textrm{Entries in row }3\big)}&=2\lambda_3,\\
    &\shortvdotswithin{=}
    \sum{\big(\textrm{Entries in row }n\big)}&=(n-1)\lambda_{n}.\\
\end{align*}
Summing all values in the Young diagram of $\lambda$ yields $\sum_{i=1}^{n}(i-1)\lambda_{i}$. By subtracting this from ${\binom{n}{2}}$ we have
\[\binom{n}{2}-\sum{\big(\textrm{Entries in Young Diagram }i\big)}=\binom{n}{2}-\sum_{i=1}^{n}{(i-1)\lambda_i}
    =g(\lambda).\]

We can calculate ${e_2(\lambda)}$ similarly by forming a Young diagram of shape $\lambda$ where each each box's entry counts the number of boxes in the same row that are strictly to the left of its own. Again using ${(4,3,1,1)}$ as an example, we would have
\[\young(0123,012,0,0)\,.\]
In general, the ${i^{\textrm{th}}}$ row  of the diagram for $\lambda$ will be of the form 

\[\ytableausetup{mathmode, boxsize=2em}
\begin{ytableau}
\scriptstyle{0} & \scriptstyle{1} & \none[\dots] & \scriptstyle{\lambda_i-2} & \scriptstyle{\lambda_i-1} \\
\end{ytableau}\,,\]

so the sum of the boxes in the ${i^{\textrm{th}}}$ row will be ${\binom{\lambda_i}{2}}$. Summing all of the entries in the Young diagram of $\lambda$ and subtracting this from ${\binom{n}{2}}$ yields
\[\binom{n}{2}-\sum_{i=1}^{n}{\big(\textrm{Entries in row }i\big)}=\binom{n}{2}-\sum_{i=1}^{n}\binom{\lambda_i}{2}=e_2(\lambda),\]
where the last equality is by Lemma \ref{lemma1}. Since ${g(\lambda)}$ is calculated by counting boxes in the columns of the Young diagram of $\lambda$, and ${e_2(\lambda)}$ is calculated by counting boxes in the rows, it follows that $g(\lambda)=e_2(\widetilde{\lambda})$.

\end{proof}

\begin{proposition}
Let $\lambda$ and $\mu$ be partitions of $n$. If $\mu \prec \lambda$ then $g(\mu)<g(\lambda)$ and $e_{2}(\lambda)<e_{2}(\mu)$.
\label{prop3}
\end{proposition}

The normalized Gini index on $\mathbb{R}^n$, discussed in \cite{Inequalities}, is equal to $\frac{2g}{n^2}$ when restricted to $P_n$. This function is known to be strictly Schur convex, so Proposition \ref{prop3} can be partially deduced from this fact. A complete proof of Proposition \ref{prop3} that does not utilize these facts is presented below.

\begin{proof}
Let $\lambda=(\lambda_1,\lambda_2,\ldots,\lambda_n)$ and $\mu=(\mu_1,\mu_2,\ldots,\mu_n)$ be partitions of $n$ (padded with zeros in their tails, if necessary). Suppose that $\lambda$ covers $\mu$, I.E. there is no partition $\rho$ of $n$ such that $\mu \prec \rho \prec \lambda$. Now $\lambda$ covers $\mu$ if and only if 
\begin{align*}
    \lambda_i &= \mu_i +1,\\
    \lambda_k &= \mu_k -1, and\\
    \lambda_j &= \mu_j,
\end{align*}
for all $j \neq i$ or $k$, and either $k=i+1$ or $\mu_i=\mu_k$ \cite{Brylawski}. In other words, $\lambda$ covers $\mu$ if and only if all but two of the rows (row $i$ and $k$, with $i<k$) in the Young diagrams of $\lambda$ and $\mu$ are of the same length, and the diagram of $\lambda$ can be obtained from that of $\mu$ by removing the last box from the $k^{\textrm{th}}$ row, and appending it to end of the $i^{\textrm{th}}$ row.

Begin with the Young diagram of $\mu$ and, as in the proof of Proposition \ref{prop2}, fill the diagram with numbers so that each box's entry counts the number of boxes weakly to the left of it.
\[\ytableausetup
{mathmode, boxsize=2em}
\begin{ytableau}
\scriptstyle{0} & \scriptstyle{1} & \none[\dots] & \scriptstyle{\mu_1-4} & \scriptstyle{\mu_1-3} & \scriptstyle{\mu_1-2} & \scriptstyle{\mu_1-1} \\
\none[\vdots]\\
\scriptstyle{0} & \scriptstyle{1} & \none[\dots] & \scriptstyle{\mu_i-2} & \scriptstyle{\mu_i-1} \\
\none[\vdots]\\
\scriptstyle{0} & \scriptstyle{1} & \none[\dots] & \scriptstyle{\mu_k-2} & \scriptstyle{\mu_k-1} \\
\none[\vdots]\\
\scriptstyle{0} & \scriptstyle{1} & \none[\dots] & \scriptstyle{\mu_n-1} \\
\end{ytableau}\]
From row $k$ we remove the box containing $\mu_k-1$ and append it to the end of row $i$ to obtain a diagram of shape $\lambda$.
\[\ytableausetup
{mathmode, boxsize=2em}
\begin{ytableau}
\scriptstyle{0} & \scriptstyle{1} & \none[\dots] & \scriptstyle{\mu_1-4} & \scriptstyle{\mu_1-3} & \scriptstyle{\mu_1-2} & \scriptstyle{\mu_1-1} \\
\none[\vdots]\\
\scriptstyle{0} & \scriptstyle{1} & \none[\dots] & \scriptstyle{\mu_i-2} & \scriptstyle{\mu_i-1} & \scriptstyle{\mu_k-1}\\
\none[\vdots]\\
\scriptstyle{0} & \scriptstyle{1} & \none[\dots] & \scriptstyle{\mu_k-2}  \\
\none[\vdots]\\
\scriptstyle{0} & \scriptstyle{1} & \none[\dots] & \scriptstyle{\mu_n-1} \\
\end{ytableau}\]
But $i < k$, hence $\mu_k-1 \leq \mu_i-1$, and the corresponding filling of the diagram for $\lambda$ would have the last cell in row $i$ containing $\mu_i$, which is strictly greater than $\mu_k-1$. Thus the sum of all numbers in the diagram for $\lambda$ is 
\[\sum_{\substack{j=1 \\ j\neq i,k}}^{n}{\binom{\mu_j}{2}}+{\binom{\mu_i+1}{2}} + {\binom{\mu_k-1}{2}}, \]
and the sum of all numbers in the diagram for $\mu$ is
\[\sum_{j=1}^{n}{\binom{\mu_j}{2}}. \]
By Lemma \ref{lemma1}, we have
\begin{align*}
    e_2(\mu)-e_2(\lambda) =& \,\binom{n}{2}-\sum_{j=1}^{n}{\binom{\mu_j}{2}}-\\
    &\,\left( \binom{n}{2}-\left(\sum_{\substack{j=1 \\ j\neq i,k}}^{n}{\binom{\mu_j}{2}}+{\binom{\mu_i+1}{2}} + {\binom{\mu_k-1}{2}} \right)\right)\\
    =&\,{\binom{\mu_i+1}{2}}+{\binom{\mu_k-1}{2}}-{\binom{\mu_i}{2}}{\binom{\mu_k}{2}}\\
    =&\,\frac{(\mu_i)(\mu_i+1-\mu_i+1)+(\mu_k-1)(\mu_k-2-\mu_k+1)}{2}\\
    =&\,\frac{2\mu_i+1-\mu_k}{2}\\
    >&\,0.
\end{align*}
So $e_2(\mu)>e_2(\lambda)$. Moreover $\mu \prec \lambda$ if and only if $\widetilde{\lambda} \prec \widetilde{\mu}$. Hence $\widetilde{\mu}$ covers $\widetilde{\lambda}$, and by Proposition \ref{prop2}, $g(\mu)< g(\lambda).$ The general case follows by transitivity.
\end{proof}

\section{Antichains in the Dominance Order}
The converse of Proposition \ref{prop3} does not hold, in general, as seen in Example \ref{ex}.
\begin{example}
\label{ex}
Let $\lambda=(5,5)$ and $\mu=(6,2,2)$ be partitions of $n=10$. Then
\begin{align*}
    g(\mu)=39&<g(\lambda)=40,\text{ and}\\
    e_2(\lambda)=25&<e_2(\mu)=28,
\end{align*}
but $\mu\nprec\lambda$.
\end{example}
The contrapositive of Propasition \ref{prop3}, however, provides us with an easily calculated lower bound on the width of $P_n$. 
\begin{corollary}(Contrapositive to Prop. \ref{prop3})
Let $\lambda\neq \mu$ be partitions of $n$. If
\begin{align*}
    g(\mu)&= g(\lambda),\\
    e_2(\mu)&=e_2(\lambda),\\
    g(\mu)&>g(\lambda)\text{ and }e_2(\mu)>e_2(\lambda),\text{ or}\\
    g(\mu)&<g(\lambda)\text{ and }e_2(\mu)<e_2(\lambda),
\end{align*}
then $\lambda$ and $\mu$ are incomparable. 
\end{corollary}
In other words, we can find lower bounds on the size of the maximal level set of $g$ on $P_n$. As we will see shortly, these lower bounds can be calculated using the generating function of Proposition \ref{prop4}.

It is often useful in Algebraic Combinatorics to record a discrete data set in the coefficients or powers of a formal power series. We call these power series ``generating functions'' for the data set. By ``formal'' we mean that the convergence of the series is immaterial. Any variables appearing in the series are taken as indeterminates rather than numbers. Alternatively, one may consider a formal power series as an ordinary power series that converges only at zero.

We define a generating function  for the Gini index $g(\lambda)$ of an integer partition $\lambda$ by
\[G(q,x)=\sum_{n=1}^{\infty}{\sum_{\lambda \vdash n}{q^{\left(\binom{n+1}{2}-g(\lambda)\right)}x^n}}.\]

Perhaps the most widely known example of a generating function is that of the integer partition function, ${P(n)}$, which counts the number of partitions of the integer $n$. 

\begin{example}
${n=5}$ has partitions

\begin{center}
    $(1,1,1,1,1)$,
    $(2,1,1,1)$,
    $(2,2,1)$,
    $(3,1,1)$,
    $(3,2)$,
    $(4,1)$, and
    $(5)$,
\end{center}
so ${P(5)=7.}$ 
\end{example}
It is well known (see, for example, \cite{Andrews}) that ${P(n)}$ has generating function
\[\prod_{n=1}^{\infty}{\frac{1}{1-x^n}}=\sum_{n=0}^{\infty}{P(n)x^n},\]
Where $P(0)$ is defined to be $1$.

In light of our previous results, we obtain a similar equality for ${G(q,x)}$.
\begin{proposition}
\[\prod_{n=1}^{\infty}{\frac{1}{1-q^{\binom{n+1}{2}}x^n}}-1=\sum_{n=1}^{\infty}{\sum_{\lambda \vdash n}{q^{\left(\binom{n+1}{2}-g(\lambda)\right)}x^n}}\]
\label{prop4}
\end{proposition}

\begin{proof}
We will show that the power series about $x=0$ of the product \begin{equation}
    {\prod_{n=1}^{\infty}{\frac{1}{1-q^{\binom{n+1}{2}}x^n}}}-1
    \label{eq1}
\end{equation} has as its general coefficient \[{\sum_{\lambda\vdash n}{q^{\left(\binom{n+1}{2}-g(\lambda)\right)}}}.\]
Considering each factor of the product as a geometric series, we have

\begin{align*}
    \prod_{n=1}^{\infty}{\frac{1}{1-q^{\binom{n+1}{2}}x^{n}}}=
    &
    \frac{1}{\left(1-q^{\binom{2}{2}}x\right)}
    \cdot
    \frac{1}{\left(1-q^{\binom{3}{2}}x^2\right)}
    \cdot
    \frac{1}{\left(1-q^{\binom{4}{2}}x^3\right)}
    \cdot
    \frac{1}{\left(1-q^{\binom{5}{2}}x^4\right)}
    \cdot
    \cdots
    \\
    =&
    \left(1+q^{\binom{2}{2}}x+q^{2{\binom{2}{2}}}x^2+q^{3{\binom{2}{2}}}x^3+q^{4{\binom{2}{2}}}x^4+\cdots \right)\cdot
    \\
    &
    \left(1+q^{\binom{3}{2}}x^2+q^{2{\binom{3}{2}}}x^4+q^{3{\binom{3}{2}}}x^6+q^{4{\binom{3}{2}}}x^8+\cdots \right)\cdot
    \\
    &
    \left(1+q^{{\binom{4}{2}}}x^3+q^{2{\binom{4}{2}}}x^6+q^{3{\binom{4}{2}}}x^9+q^{4{\binom{4}{2}}}x^{12}+\cdots \right)\cdot
    \\
    &
    \left(1+q^{{\binom{5}{2}}}x^4+q^{2{\binom{5}{2}}}x^8+q^{3{\binom{5}{2}}}x^{12}+q^{4{\binom{5}{2}}}x^{16}+\cdots \right)\cdot\cdots.
\end{align*}

If we distribute and simplify, for example, the coefficient of $x^4$, we see that it is
\[ q^{4{\binom{2}{2}}}+q^{2{\binom{3}{2}}}+q^{{\binom{2}{2}}+{\binom{4}{2}}}+q^{2{\binom{2}{2}}+{\binom{3}{2}}}+q^{{\binom{5}{2}}},  \]
where each of the terms correspond to the partitions
\[(1,1,1,1),\,(2,2),\,(3,1),\,(2,1,1),\,\textrm{and}\,(4),\]
respectively, by 
\[
    (\lambda_1,\lambda_2,\ldots,\lambda_l)\mapsto q^{\left(\sum_{i=1}^{l}{\binom{\lambda_i+1}{2}}\right)}.
\]
This is true, in general, for the coefficient of $x^n$, for all positive integers $n$. To see this, consider the coefficient on $x^n$ in the power series expansion of equation \ref{eq1}. If we set $q=1$ in the product of equation \ref{eq1}, we obtain the generating function of $P(n)$ (the number of partitions of $n$). Hence there are $P(n)$ different ways to obtain a power of $x^n$. So the $x^n$ term in equation \ref{eq1} will be of the form
\begin{equation*}
    \sum_{j=1}^{P(n)}\prod _{i=1}^{m_j}q^{\left(a_{j,i}{\binom{\lambda_{j,i}+1}{2}}\right)}x^{\left(a_{j,i}\lambda_{j,i}\right)},
\end{equation*}
where $a_{j,i},\lambda_{j,i}>0$, and $\sum_{i=1}^{m_j}a_{j,i}\lambda_{j,i}=n$. Thus the coefficient on $x^n$ will be
\begin{equation}
    \sum_{j=1}^{P(n)}q^{\left(\sum_{i=1}^{m_j}{ a_{j,i}{\binom{\lambda_{j,i}+1}{2}} }\right) }.
    \label{eq2}
\end{equation}
Since each ${\binom{\lambda_{j,i}+1}{2}}$ in equation \ref{eq2} comes from a different term of the product in equation \ref{eq1}, we have that $\lambda_{j,i}\neq\lambda_{j,k}$ whenever $i\neq k$. Therefore, by reordering, we may choose the power $a_{j,i}{\binom{\lambda_{j,i}+1}{2}}$ on $q$ so that $\lambda_{j,i}>\lambda_{j,{i+1}}> 0$, for $1\leq i < m_j$. It follows that $(\lambda_{j,1}^{a_{j,1}},\lambda_{j,2}^{a_{j,2}},\ldots,\lambda_{j,{m_j}}^{a_{j,{m_j}}})$ is a partition of $n$, where $\lambda_{j,i}$ is repeated $a_{j,i}$ times.

Again, using the generating function for $P(n)$, the ways of writing $x^n$ as a product $\prod_{i}x^{\left( a_{j,i}\lambda_{j,i} \right)}$ (where $a_{j,i}\lambda_{j,i}>0$) is in bijection with the partitions of $n$. Since each of the sums $\sum_{i=1}^{m_j}a_{j,i}\lambda_{j,i}=n$ have distinct summands for all $1\leq j \leq P(n)$, it follows that the sums $\sum_{i=1}^{m_j}a_{j,i}{\binom{\lambda_{j,i}+1}{2}}$ are all distinct for different values of $j$. In other words, every partition $( \lambda_{j,1}^{a_{j,1}},\ldots,\lambda_{j,{m_j}}^{a_{j,{m_j}}} )$ of $n$ appears as a power in equation \ref{eq2}. Hence equation \ref{eq2} is equal to
\begin{equation*}
 \sum_{\lambda\vdash n} q^{\left(\sum_{i=1}^{n}\binom{\lambda_i+1}{2}\right)}.
\end{equation*}
By Lemma \ref{lemma1}, $e_2(\lambda)=\binom{n+1}{2}-\sum_{i=1}^{n}\binom{\lambda_i+1}{2}$, thus 
\begin{equation*}
    \sum_{\lambda\vdash n} q^{\left(\sum_{i=1}^{n}\binom{\lambda_i+1}{2}\right)}=\sum_{\lambda\vdash n} q^{\left(\sum_{i=1}^{n} \binom{n+1}{2}-e_2(\lambda)\right)}.
\end{equation*}
Finally, by Proposition \ref{prop2}, we have that the general coefficient on $x^n$ in the power series expansion of equation \ref{eq1} is
\begin{equation*}
    \sum_{\lambda\vdash n} q^{\left(\sum_{i=1}^{n} \binom{n+1}{2}-g(\lambda)\right)}.
\end{equation*}
\end{proof}

We can use $G(q,x)$ to find lower bounds on the width of $P_n$ by calculating the cardinalities of the maximum level sets of $g$ on $P_n$. In particular, the ``size'' of these level sets will be the largest coefficient on the powers of $q$ that form the coefficient of $x^n$. Expanding $G(q,x)$ yields
\begin{align*}
        \sum_{n=1}^{\infty}\sum_{\lambda \vdash n}q^{\left(\binom{n+1}{2}-g(\lambda)\right)}x^{n}
        &=qx+(q^2+q^3)x^2+(q^3+q^4+q^6)x^3\\
        &+(q^4+q^5+q^6+q^7+q^{10})x^4\\
        &+(q^5+q^6+q^7+q^8+q^9+q^{11}+q^{15})x^5\\
        &+(\cdots+{2}q^9+\cdots)x^6+(\cdots+2q^{10}+\cdots)x^7\\
        &+(\cdots+2q^{11}+\cdots)x^8+(\cdots+3q^{15}+\cdots)x^9\\
        &+\cdots.
    \end{align*}
So the size of the maximal level sets of $g$ are 1, 1, 1, 1, 1, 2, 2, 2, and 3, on $P_1$ through $P_9$, respectively.

Let $b(n)$ denote the size of the maximal level set of $g$ on $P_n$ (A337206 in \cite{OEIS}). Proposition \ref{prop3} implies that $b(n)\leq a(n)$ for all positive integers $n$, where $a(n)$ is the size of the maximum antichain in $P_n$ (A076269 in \cite{OEIS}). There are known asymptotic bounds on $a(n)$. The best known bounds follow from Dilworth's Theorem; since $P(n)$ (the number of partitions of $n$) is clearly an upper bound on $a(n)$, we have that $a(n)\geq P(n)/(h(P_n)+1)$, where $h(P_n)$ is the length of a maximal chain in $P_n$. This argument was applied in \cite{Early} without proof, so we will sketch the proof and use the same reasoning to acquire an asymptotic lower bound on the sequence $b(n)$. Formulas for the sequence $h(P_n)$ have been known for some time, but the asymptotics were not fully understood until Greene and Kleitman (\cite{Greene} ) proved that
\[h(P_n)\sim (2n)^{3/2}/3.\] Combining this with the famous result of Hardy and Ramanujan ( \cite{Hardy} ), 
\[P(n)\sim\frac{e^{\pi\sqrt{2n/3}}}{4n\sqrt{3}}, \]
we see that
\[ \Omega\left( \frac{e^{\pi\sqrt(2n/3)}}{n^{5/2}} \right)\leq a(n) \leq O\left( \frac{e^{\pi\sqrt{2n/3}}}{n}\right). \]

We can similarly obtain a lower bound on $b(n)$. Observe that $g$ may take on values between $0$ and $\binom{n}{2}$. As such, the level sets $b(n)$ of $g$ on $P_n$ are bounded below by
\[ \frac{P(n)}{\binom{n}{2}}\leq b(n). \]
Applying the same formulas as above, yields the following.
\begin{proposition}
\[ \Omega\left(\frac{e^{\pi\sqrt{2n/3}}}{n^3}\right)\leq b(n), \]
for all $n>1$.
\label{Gini Bound}
\end{proposition}
While not as ``sharp'' as the bound obtained via Dilworth's Theorem, Proposition \ref{Gini Bound} shows that the level sets of the Gini index provide a good lower bound on the size of the maximum antichain in $P_n$. 

\section{Expected Value of the Gini Index on $P_n$}
With the results from the previous section, a natural follow-up question would be, ``What is the expected value of the Gini index?'' To formalize this question, we will view $g$ as a real-valued discrete random variable with sample space $P_n$. To each outcome $m\in\left[0,\binom{n}{2}\right]$, we assign the probability
\[ \textbf{P}(g= m)=\frac{|g^{-1}(m)|}{P(n)}.\]
That is, the probability that $g=m$ is the number of partitions $\lambda\in P_n$ such that $g(\lambda)=m$, divided by the number of partitions of $n$. For a fixed valued of $n\in\N$, computing the expected value of the Gini index is as simple as calculating the partitions of $n$ (which, in reality, is not at all simple). 
\begin{example}
If $n=6$, the partitions of $n$ are
\begin{center}
    (1,1,1,1,1,1),\\
    (2,1,1,1,1),\\
    (2,2,1,1),\\
    (2,2,2),\\
    (3,1,1,1),\\
    (3,2,1),\\
    (3,3),\\
    (4,1,1),\\
    (4,2),\\
    (5,1)\text{, and}\\
    (6).
\end{center}
Their function values under $g$ are
\begin{align*}
    g((1^6))&=0,\\
    g((2,1^4))&=5,\\
    g((2^2,1^2))&=8,\\
    g((2^3))&=9,\\
    g((3,1^3))&=9,\\
    g((3,2,1))&=11,\\
    g((3^2))&=12,\\
    g((4,1^2))&=12,\\
    g((4,2))&=13,\\
    g((5,1))&=14,\text{ and}\\
    g((6))&=15.
\end{align*}
\end{example}
Let \textbf{E}$(g,P_n)$ represent the expected value of $g$ on $P_n$. By the above example, we have that
\[\textbf{E}(g)=\frac{0}{11}+\frac{5}{11}+\frac{8}{11}+2\frac{9}{11}+\frac{11}{11}+2\frac{12}{11}+\frac{13}{11}+\frac{14}{11}+\frac{15}{11}= \frac{108}{11}\approx 9.8. \]
If $\lambda\vdash n$, then $g(\lambda)<n$ only when $\lambda=(1^n)$. Hence the expected value of the Gini index on $P_n$ tends to infinity as $n$ approaches infinity. A natural followup question is, ``What is the expected value of the Gini index on $P_n$ when normalized by its maximum value, and what is its end behavior?'' In our example of $n=6$, the ``normalized'' expected value of $g$ is obtained by dividing by the maximum value of $15$:
\[\frac{0}{165}+\frac{5}{165}+\frac{8}{165}+2\frac{9}{165}+\frac{11}{165}+2\frac{12}{165}+\frac{13}{165}+\frac{14}{165}+\frac{15}{165}= \frac{36}{55}\approx 0.655. \] 

To determine the limit of these numbers (if it exists), we recall the generating function of $g$ from Proposition \ref{prop4}: 
\[ \sum_{n=1}^{\infty}\sum_{\lambda\vdash n}q^{\left(\binom{n+1}{2}-g(\lambda)\right)}x^n=\prod_{n=1}^{\infty}\frac{1}{1-q^{\binom{n+1}{2}}x^n}-1. 
\]
Taking the formal partial derivative  with respect to $q$, we have
\[
    \sum_{n=1}^{\infty}\sum_{\lambda\vdash n}\left( \binom{n+1}{2}-g(\lambda) \right)q^{\binom{n+1}{2}-g(\lambda)-1}x^n=\sum_{n=1}^{\infty}\left(\frac{\binom{n+1}{2}q^{\binom{n+1}{2}-1}x^n}{\left( 1-q^{\binom{n+1}{2}}x^n\right)^2}\cdot \sum_{\substack{i=1\\i\neq n}}^{\infty} \frac{1}{1-q^{\binom{i+1}{2}}x^i} \right). 
    \]
Setting $q=1$ yields the following:
\begin{align*}  
\sum_{n=1}^{\infty}\sum_{\lambda\vdash n}\left( \binom{n+1}{2}-g(\lambda) \right)x^n&=\sum_{n=1}^{\infty}\left(\frac{\binom{n+1}{2}x^n}{\left( 1-x^n\right)^2}\cdot \prod_{\substack{i=1\\i\neq n}}^{\infty} \frac{1}{1-x^i} \right)\\
&=\left(\sum_{n=1}^{\infty}\frac{\binom{n+1}{2}x^n}{\left( 1-x^n\right)}\right)\cdot\left( \prod_{i=1}^{\infty} \frac{1}{1-x^i} \right).
\end{align*}
We have thus obtained a generating function for the sums of the outcomes of $g$ on $P_n$, for any $n\in\N$. To determine the expected value of the normalized Gini index on $P_n$, one need only divide the coefficient of $x^n$ in the series by 
\[ \binom{n}{2}\cdot P(n); \]
the maximum value of $g$ on $P_n$ times the number of partitions of $n$. To find the general form for these coefficients, however, we will need to write the expression on the left as a formal power series.

We can recognize the right hand term of the above product as the generating function of the partition function, $P(i)$. Moreover, the left hand term of the product is a Lambert series, which formally sums to
\[  
\sum_{n=1}^{\infty}\frac{\binom{n+1}{2}x^n}{\left( 1-x^n\right)}=\sum_{n=1}^{\infty}\sum_{d| n}\binom{d+1}{2}x^n.
\]
Manipulating the sum on the right, we obtain
\begin{align*}
    \sum_{n=1}^{\infty}\sum_{d| n}\binom{d+1}{2}x^n&=\sum_{n=1}^{\infty}\sum_{d|n}\frac{1}{2}(d^2+d)x^n\\
    &=\frac{1}{2}\sum_{n=1}^{\infty}(\sigma_1(n)+\sigma_2(n))x^n,
\end{align*}
where
\[\sigma_x(n)=\sum_{d|n}d^x\]
is the so-called \emph{sum of divisors} function, for $x\in\mathbb{C}$.

Putting this all together yields
\begin{align*}
    \sum_{n=1}^{\infty}\sum_{\lambda\vdash n}\left( \binom{n+1}{2}-g(\lambda) \right)x^n&=\left( \frac{1}{2}\sum_{n=1}^{\infty}(\sigma_1(n)+\sigma_2(n))x^n\right)\cdot\left( \sum_{i=0}^{\infty}P(i)x^i \right)\\
    &=\frac{1}{2}\sum_{n=1}^{\infty}\sum_{i=0}^{n-1}P(i)(\sigma_1(n-i)+\sigma_2(n-i))x^n.
\end{align*}
By distributing the $x^n$ on the left we can isolate the sum containing $g(\lambda)$, obtaining
\begin{align*}
    \sum_{n=1}^{\infty}\sum_{\lambda\vdash n}g(\lambda)x^n&=\sum_{n=1}^{\infty}P(n)\binom{n+1}{2}x^n-\frac{1}{2}\sum_{n=1}^{\infty}\sum_{i=0}^{n-1}P(i)(\sigma_1(n-i)+\sigma_2(n-i))x^n\\
    &=\sum_{n=1}^{\infty}\frac{1}{2}\left(P(n)(n^2+n)- \sum_{i=0}^{n-1}P(i)(\sigma_1(n-i)+\sigma_2(n-i)) \right)x^n.
\end{align*}
We summarize these results in the following theorem.
{\begin{theorem}
Let $n>1$. The expected value of the normalized Gini index on the set $P_n$ of partitions of $n$ is
\[ \textbf{E}(g,P_n)=\sum_{\lambda\vdash n}\frac{g(\lambda)}{\binom{n}{2}P(n)}=\frac{n+1}{n-1}-\sum_{i=0}^{n-1}\frac{P(i)(\sigma_1(n-i)+\sigma_2(n-i))}{P(n)(n^2-n)}, \]
where $P(n)$ is the number of partitions of $n$.
\label{expected value prop}
\end{theorem}}

Table \ref{Expected Value Table} contains the expected values of the Gini index on $P_n$ for select values of $n$.

\begin{center}
\begin{tabular}{
||P{3cm}|P{5cm}|P{5cm}||  }
 \hline
 $n$ & Expected Value (Exact) & Expected Value (Decimal)\\
 \hline
 $5$&$22/55$&$0.6286$\\
 $10$&$451/630$&$0.7159$\\
 $15$&$14023/18480$&$0.7588$\\
 $20$&$93647/119130$&$0.7861$\\
 $25$&$31541/39160$&$0.8054$\\
 $30$&$1999303/2437740$&$0.8201$\\
 $35$&$7366148/8855385$&$0.8318$\\
 $40$&$628333/746760$&$0.8414$\\
 $45$&$340724/401103$&$0.8495$\\
 $50$&$42848633/50035370$&$0.8564$\\
 $100$&$-$&$0.8951$\\
 $500$&$-$&$0.9510$\\
 $1000$&$-$&$0.9649$\\
 $5000$&$-$&$0.9841$\\
 $10000$&$-$&$0.9887$\\
 
 \hline
\end{tabular}
\captionof{table}{Expected value of $g_1$ on $P_n$}\label{Expected Value Table}
\end{center}
The expected value is monotone increasing at all intermediate values.\\

One of the questions at the start of this section regarded the asymptotic behavior of the expected value in Theorem \ref{expected value prop}. Based on the experimental data in the above table, we make the conservative conjecture that
\[\lim_{n\rightarrow \infty}\textbf{E}(g,P_n)=1. \]

\section{Further Generalizations}
    So far we have been working with the Gini index $g$ defined on the set $P_n$ of partitions of a positive integer $n$. This corresponds to the ``real life'' situation in which $n$ dollars are distributed amongst $n$ people. we can broaden our notion of the Gini index of a partition by considering distributions of the following nature.\\
    
    Suppose we have 6 dollars distributed amongst three people, where one person gets 4 dollars, one gets 2, and the last gets 0. This distribution corresponds to the partition $\lambda=(4,2,0)$ of $6$. The most equitable distribution would be that in which all three people have the same amount; that is, the partition $\mu=(2,2,2)$. It is clear that such distributions are in bijection with those partitions of $6$ with at most $3$ parts. \\
    
    To find the Lorenz Curve of $\lambda$, we pad its tail with zeros until it has $6$ parts
    \[ \lambda=(4,2,0,0,0,0). \]
    Doing the same with $\mu$, we obtain the graph in Figure \ref{Gini Figure} .
    
    \begin{figure}[htp]
    \centering
    \includegraphics[width=10cm]{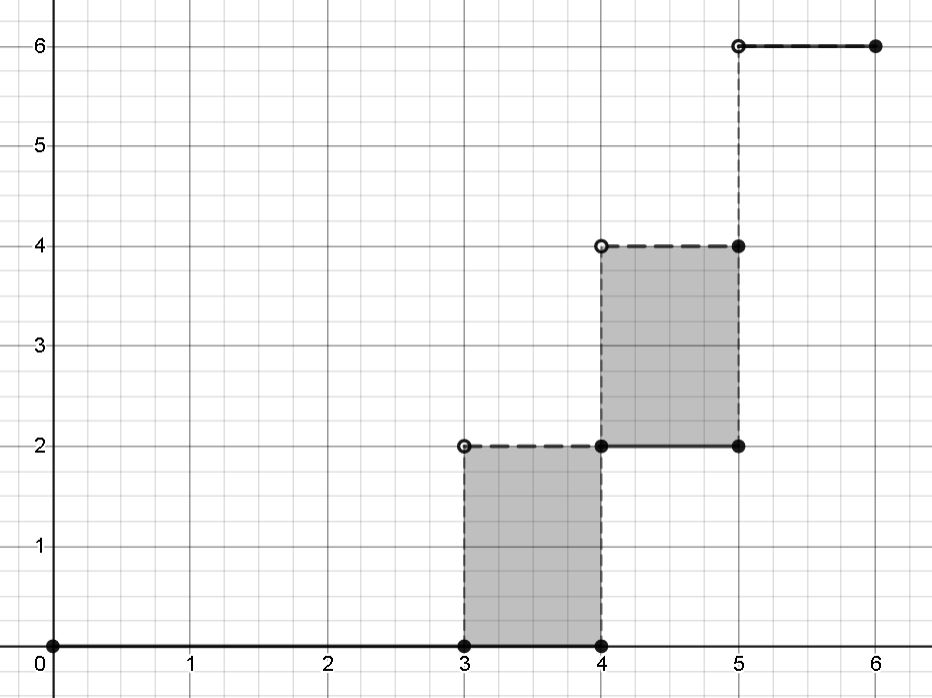}
    \caption{The line of equality (dashed), the Lorenz curve of the partition (4,2,0) of 6 (solid), and the area between them (shaded).}
    \label{Gini Figure}
\end{figure}
    
    It would be sensible to define the Gini index of such a distribution to be the area between the new line of equality and the Lorenz Curve of $\lambda$, which we see from the graph is $4$. Noting that the line of equality and the Lorenz curve of $\lambda$ coincide on the interval $[0,3]$, we may generalize our notion of a Lorenz curve as follows.\\
    
    Let $n,k\in\N$. Suppose we have $n$ individuals in a population, amongst whom is distributed $nk$ dollars. Let $\lambda=(\lambda_1,\ldots,\lambda_n)$ be the corresponding partition of $nk$ with $n$ parts (padded with zeros on the tail if necessary). The most equitable distribution then corresponds to the partition $(k^n)=(k,k,\ldots,k)$ (with $n$ parts), which we call the line of equality. 
   
    The \emph{Lorenz curve}, $L_{nk,n,\lambda}:[0,n]\longrightarrow[0,nk]$, is then defined by setting $L_{nk,n,\lambda}(0)=0$, and $L_{nk,n,\lambda}(x)=\sum_{i=n-j+1}^n\lambda_i$, where $1\leq j \leq n$ is the unique positive integer such that $x\in (j-1,j]$. In other words, if $x\in(j-1,j]$, then \[L_{nk,n,\lambda}(x)=\lambda_n+\lambda_{n-1}+\cdots+\lambda_{n-j+1}. \]
    Using this definition, we see that the Lorenz curve of $(k^n)$, the ``line of equality'', is given by
    \[ y=k\lceil x\rceil. \]
    This definition clearly restricts to the previous definition of a Lorenz curve when $k=1$. In this event, we will simplify our notation to $L_{n,n,\lambda}=L_{\lambda}$.
    
    To generalize the notion of the Gini index, we note that the formulas from section 2.3 still hold in this setting. That is, the area under the Lorenz curve $L_{nk,n,\lambda}$ is given by
    \[ \sum_{i=1}^n i\lambda_i. \]
    The area between the line of equality and the Lorenz curve of $\lambda$ is then
    \[\sum_{i=1}^n in-\sum_{i=1}^n i\lambda_i=\sum_{i=1}^n (i-1)k-\sum_{i=1}^n(i-1)\lambda_i.  \]
    These sums occur frequently throughout this paper, so we will adopt the function notation
    \[ b(\lambda)=\sum_{i=1}^n (i-1)\lambda_i.  \]
    As expected, we then define the Gini index of $\lambda$ to be
    \[g_{nk,n}(\lambda)=b(k^n)-b(\lambda).  \]
    If $k=1$, then we obtain the previous definition of the Gini index, and simplify our notation to $g_{n,n}(\lambda)=g(\lambda)$.

    This generalization, $g_{nk,n}$, of the Gini index is featured prominently in Chapter 6.

\newpage
\chapter{Kostka-Foulkes Polynomials}
The Gini index, $g$ of Chapter 3 is closely related to the representation theory of the symmetric group, $S_n$. These connections are discussed at length in Chapter 5. The construction that bridges the gap between the combinatorics we have seen thus far, and the representation theory we will encounter later, is that of the so-called ``Kostka-Foulkes'' polynomials.
\section{Kostka Numbers}
Let $\lambda$ be a partition of $n$. Recall that a semistandard tableau of shape $\lambda$ is a filling of the Young diagram of $\lambda$ by positive integers in $[n]$ that is
\begin{enumerate}
    \item weakly increasing across each row, and
    \item strictly increasing down each column.
\end{enumerate}
To each tableaux, $T$, of shape $\lambda\vdash n$, there is an associated $n$-tuple of non-negative integers $\mu=(\mu_1,\ldots,\mu_n)$ called the \emph{weight} of $T$, where each $\mu_i$ records the number of times that $i$ appears as an entry in $T$. We will be especially interested in tableaux whose weights are partitions.

\begin{example}
The tableau
\[\young(122,23,3,4)\]
of shape $\lambda=(3,2,1,1)$ has content $(1,3,2,1)$, and is a semistandard tableau. The  tableau
\[\young(125,34,6,7)\]
has shape $\lambda$, and weight $(1^7)$, and is a standard tableau.
\end{example}
As illustrated in the previous example, a semistandard tableau is standard if and only if its weight is $(1^n)$.

A natural question to ask at this point would be, ``what is the number of (semi) standard tableaux of shape $\lambda$ (and weight $\mu$)? The answer when $\lambda$ has weight $(1^n)$ is given by the so-called hook length formula. Each box in a Young diagram $\lambda$ determines a \emph{hook} - which consists of the boxes weakly below, and weakly to the right of that box. The \emph{hook length} of a box is the number of boxes in its hook.
\begin{example}
In the diagram
\[\young(531,31,1),  \]
each box is labeled with its hook length.
\end{example}

\begin{theorem}(Frame, Robinson, and Thrall)
Let $\lambda$ be a partition of $n$. The number, $f^{\lambda}$ of standard tableaux with shape $\lambda$ is 
\[f^{\lambda}=\frac{n!}{\prod_{i,j}h_{\lambda}(i,j) }.  \]
Where $h_{\lambda}(i,j)$ is the hook length of the box in the $i^{\text{th}}$ row and $j^{\text{th}}$ column of $\lambda$.
\label{Hook Length Formula}
\end{theorem}
Let $\lambda$ and $\mu$ be partitions of $n$. We define the non-negative integer $K_{\lambda, \mu}$ to be the number of semistandard tableaux of shape $\lambda$ and weight $\mu$. The number $K_{\lambda, \mu}$ is called a \emph{Kostka number}.
The Kostka numbers, $K_{\lambda, \mu}$, appear frequently throughout representation theory -- perhaps most famously as the multiplicity of the weight $\mu$ in an irreducible representation of $\mathfrak{gl}_n(\mathbb{C})$ with highest weight $\lambda$ (also the multiplicity of the weight $\mu$ in the polynomial irrep of $GL_n(\mathbb{C})$ with highest weight $\lambda$).

The Kostka numbers have a generalization called the Kostka-Foulkes polynomials (sometimes called the Kostka polynomials or $q$-Kostka polynomials). These polynomials encode information about the irreducible representations of the symmetric and general linear groups, and relate the Gini index, as we know it, to representation theory. They are defined in terms of the Schur polynomials, and Hall-Littlewood polynomials.

\section{Schur Polynomials}
Let $\lambda$ be a partition of $n$ with at most $\ell$ parts. The \emph{Schur polynomial}, $s_{\lambda}(x_1,\ldots,x_{\ell})$, associated to $\lambda$ is defined as 
\[s_{\lambda}=\sum_{T}x^{T},  \]
where the sum is over all semistandard tableaux $T$ of shape $\lambda$ using numbers from $[\ell]$, and the monomial $x^{T}$ is defined
\[x^T=\prod_{i=1}^{\ell}(x_i)^{\text{number of times $i$ occurs in $T$}}.  \]

\begin{example}
If $\lambda=(2,1)$ then the possible semistandard tableaux of shape $\lambda$ with numbers from $[2]$ are
\[\young(11,2),\text{ and }\young(12,2) \]
and hence the Schur polynomial, $s_{\lambda}$, in 2 variables is
\[s_{\lambda}=x_1^2x_2+x_1x_2^2.\]
\end{example}

It turns out that Schur polynomials are symmetric polynomials of degree $n$. In fact, these polynomials form an orthonormal $\mathbb{Z}$-basis for the ring of symmetric functions, $\Lambda$. Schur polynomials are ubiquitous in representation theory, appearing in the representation theory of general linear groups and symmetric groups. It is possible to define the Schur polynomials in terms of these structures as follows.
\begin{enumerate}
    \item The Schur polynomials $s_{\lambda}$, for $\lambda\vdash n$, are the images of the irreducible representations of $S_n$ under the Frobenius map.
    \item The Schur polynomials $s_{\lambda}$, for $\lambda\vdash n$, are the characters of the finite dimensional irreducible (polynomial) representations of $GL_n(\mathbb{C})$.
\end{enumerate}
For more on Schur functions, see \cite{Fulton} chapters 7 and 8,  or \cite{Macdonald} section 1.7.

\section{Hall-Littlewood Polynomials}

The Hall-Littlewood polynomials were originally defined to address a problem in group theory. If $G$ is a finite abelian $p$-group, then by the fundamental theory of finitely generated abelian groups, $G$ can be factored as
\[ G=\bigoplus_{i=1}^{\ell}\Z_{p^{\lambda_i}} \]
where, without loss of generality, $\lambda_1\geq \lambda_2\geq \cdots \geq \lambda_{\ell}>0$. The partition $\lambda=(\lambda_1,\ldots,\lambda_{\ell})$ is called the \emph{type} of $G$. 

A family of symmetric polynomials called Hall polynomials (not to be confused with Hall-Littlewood polynomials) arise from the following scenario. Given a finite abelian $p$-group $G$ of type $\lambda$, let $G_{\mu^{(1)}\ldots\mu^{(k)}}^{\lambda}(p)$ denote the number of chains of subgroups
\[<1>=H_0\triangleleft H_1\triangleleft \cdots \triangleleft H_k = G  \]
in $G$ of type $\lambda$ such that $H_i/H_{i+1}$ is of type $\mu^{(i)}$, where the $\mu^{(i)}$ are integer partitions for $1\leq i \leq k$ (\cite{Macdonald}, \cite{Hall-Littlewood}).
Hall formally introduced these numbers in \cite{Hall} and proved several important results - among them was that $G_{\mu^{(1)}\cdots\mu^{(\ell)}}^{\lambda}(p)$ is a symmetric polynomial in $p$. Moreover, the these polynomials can be used as the multiplication constants of a commutative and associate algebra $H$ called the \emph{Hall algebra}. Littlewood found in \cite{Littlewood} that the generators, $u_{\lambda}(p)$, of $H$ (indexed by all partitions $\lambda$) are of the form
\[u_{\lambda}(p)=p^{b(\lambda)}P_{\lambda}(p^{-1}),  \]
where 
\[b(\lambda)=\sum_{i=1}^{\ell}(i-1)\lambda_i,  \]
$\ell$ is the number of parts of $\lambda$, and the $P_{\lambda}(q)$ are called the \emph{Hall-Littlewood polynomials}. These polynomials are defined by
\[ P_{\lambda}(x_1,\ldots,x_n;t)=\left(\prod_{i\geq0}\prod_{j=1}^{m(i)}\frac{1-t^i}{1-t^j} \right)\sum_{\sigma\in S_n}\sigma\left( x_1^{\lambda_1}\cdots x_n^{\lambda_n}\prod_{i<j}\frac{x_i-tx_j}{x_i-x_j} \right), \]
where $m(i)$ is the number of times that $i$ occurs in $\lambda$. It turns out that the Hall-Littlewood polynomials, like the Schur polynomials, are homogeneous symmetric functions of degree $|\lambda|$, and form a $\Z$-basis for $\Lambda$. They too, like the Schur polynomials, are ubiquitous in representation theory, appearing in the character theory of finite linear groups, and projective and modular representations of symmetric groups (\cite{Hall-Littlewood} ). 

\section{Kostka Foulkes Polynomials}
The \emph{monomial symmetric functions}, \[m_{\lambda}=\sum_{\sigma\in S_n}x_1^{\sigma(\lambda_1)}\cdots x_n^{\sigma(\lambda_n)},\]
are yet another $\Z$-basis for the ring of symmetric function, $\Lambda$. Moreover, since
\[P_{\lambda}(x_1,\ldots,x_n;0)=s_{\lambda}(x_1,\ldots,x_n)  ,\]
and
\[ P_{\lambda}(x_1,\ldots,x_n;1)=m_{\lambda}(x_1,\ldots,x_n) ,\]
we see that the Hall-Littlewood polynomials interpolate between the Schur polynomials and the monomial symmetric functions.
The entries of the transition matrix from $m_{\mu}$ to $s_{\lambda}$ are the Kostka numbers, $K_{\lambda, \mu}$:
\begin{equation}
    s_{\lambda}=\sum_{\mu}K_{\lambda, \mu}m_{\mu}.
    \label{Kostka Transition}
\end{equation}
The Kostka numbers can be generalized by substituting the Hall-Littlewood polynomials into equation \ref{Kostka Transition} as follows.
\[    s_{\lambda}=\sum_{\mu}K_{\lambda, \mu}(t)P_{\mu}(x_1,\ldots,x_n;t)
\]
Here, the $K_{\lambda, \mu}(t)$ are called the \emph{Kostka-Foulkes polynomials}. 
It was conjectured by Foulkes that the polynomials $K_{\lambda, \mu}(t)$ have non-negative integer coefficients. This was eventually proven by Lascoux and Schutzenberger in \cite{Lascoux-Schutzenberger} using the notion of the charge of a partition -- their result is now given.
\begin{theorem}
\label{Lascoux & Schutzenberger}
(Lascoux and Sch\"utzenberger) Let $\lambda$ and $\mu$ be partitions of an integer $n$.
\begin{enumerate}
    \item $K_{\lambda, \mu}(t)=\sum_{T}t^{c(T)}$, where the sum is over all semistandard tableaux $T$ of shape $\lambda$ and weight $\mu$.
    \item If $\lambda \succeq \mu$, then $K_{\lambda, \mu}(t)$ is monic of degree $b(\mu)-b(\lambda)$. If $\lambda\nsucceq \mu$ then $K_{\lambda, \mu}(t)=0$.
\end{enumerate}
\end{theorem}
The function $c$ in the theorem is a combinatorial statistic known as the \emph{charge} of the tableaux $T$. The charge statistic has a somewhat complicated definition, so we opt to explain it through example.

\begin{example}
Let $T$ be the tableaux of shape $\lambda=(4,2,1)$ and weight ${\mu=(3,2,1,1)}$ given by
\[T= \young(1112,24,3) .\]
First we form the \emph{reading word} of $T$, which we denote by $w(T)$, by reading $T$ from right to left in consecutive rows, starting from the top.
\[ w(T)=2111423 \]
Next we find the \emph{standard subwords} of $T$ by finding the first $1$ in $w(T)$, and underlining it. Then we find the first $2$ in $w(T)$ occurring to the right of the $1$ that was underlined -- looping back to the beginning if necessary. Continuing in this fashion until one of each of the numbers from $1$ to $\ell(\mu)$ has been underlined (where $\ell(\mu)$ is the number of parts of $\mu$). 
\[ w(T)=2\underline{1}11\underline{4}\underline{2}\underline{3} \]
Removing the underlined numbers from $w(T)$ yields the first standard subword,
\[ w_1=1423. \]
To find the second standard subword, we perform the same procedure on the leftover numbers,
\[ 211, \]
which yields a second standard subword of
\[ w_2=21,\]
and a third standard subword of
\[ w_3=1.\]
The \emph{charge of a standard subword} is defined by the following algorithm. To find the charge of $w_1$, we mark the number $1$ with a subscript of 0. We proceed from $1$ to the right. If we encounter the number $2$ before reaching the end of $w_1$, we give it a subscript of $0$. If we have to loop around to the beginning, then we mark the $2$ with a subscript of $1$. In short, the subscript on any number counts the amount of times we must loop back to the start in order to reach that number when reading through the word from left to right, starting at the number $1$. Applying this to the standard subwords $w_1$, $w_2$ and $w_3$, we have
\[ w_1=1_0 4_1 2_0 3_0, \]
\[ w_2=2_1 1_0, \]
and
\[ w_3=1_0 .\]
The charge of these standard subwords, is the sum of their subscripts:
\[ c(w_1)=0+1+0+0=1 ,\]
\[ c(w_2)=1+0=1 ,\]
\[ c(w_3)=0 .\]
The charge of the Tableaux (or equivalently, the charge of its reading word) is then defined as the sum of the charges of its standard subwords. Hence c(T)=1+1+0=2.
\end{example}

\newpage

\section{The Degree of $K_{\lambda, \mu}(t)$}

Let $n,k\in\N$ and let $\lambda$ be a partition of $nk$ into at most $n$ parts. Recall that the Gini index $g_{nk,n}$ of $\lambda$ is given by
\[ g_{nk,n}(\lambda)=b((k^n))-b(\lambda), \]
where $b(\lambda)=\sum_{i=1}^{n}(i-1)\lambda_i$, and $(k^n)=(k,k,\ldots,k)$ is the so-called ``flat'' partition.

Since $\lambda\succeq (k^n)$, for all partitions $\lambda$ of $nk$ with at most $n$ parts, by the theorem of Lascoux and Sch\"utzenberger, we have
\begin{corollary}
The Kostka-Foulkes polynomial $K_{\lambda, (k^n)}(t)$ is monic of degree $g_{nk,n}(\lambda)$. Moreover \[g_{nk,n}(\lambda)=\max\left\{c(T):\text{$T$ is a semistandard tableaux of shape $\lambda$ and weight $(k^n)$}\right\}.  \]
\label{degree corollary}
\end{corollary}
The Kostka-Foulkes polynomials encode information about the irreducible representations of the symmetric and general linear groups, and therefore Corollary \ref{degree corollary} frames the Gini index within the context of representation theory. We will elaborate on, and explore these connections in the subsequent chapters.

\newpage

\chapter{The Gini Index and Complex Reflection Groups}

Let $n$ be a positive integer and $\lambda$ a partition of $n$. For each Specht module $S^{\lambda}$ (an irreducible representation of $S_n$) there is a special polynomial called the ``graded multiplicity'' of $S^{\lambda}$ in the coinvariants, which encodes how the coinvariant ring of $S_n$ decomposes with respect to $S^{\lambda}$. It turns out that the Kostka-Foulkes polynomial $K_{\widetilde{\lambda},(1^n)}$ is the graded multiplicity of $S^{\lambda}$ in the coinvariants. We will see that the degrees of the graded multiplicity polynomials of $S_n$ are exactly the values of the Gini index $g$ ($=g_{n,n}$) on $P_n$.\\

The symmetric group is a member of a broader family of finite groups known as ``complex reflection groups''. Using graded multiplicity polynomials, we will extend the notion of the Gini index to all other complex reflection groups, and provide formulas for the Gini index for dihedral groups.

\section{The Symmetric Group}
Let $n$ be a positive integer and $V\cong\mathbb{C}^n$ be the defining representation of $S_n$, and fix a basis $x_1,x_2,\ldots,x_n$ for $V$. In other words, if $x$ denotes the column vector $(x_1,x_2,\ldots,x_n)^t$, then $S_n$ acts on $V$ by
\[\sigma x=(x_{\sigma(1)},x_{\sigma(2)},\ldots,x_{\sigma(n)})^t, \]
where $\sigma\in S_n$. By $\mathbb{C}[V]$ we mean the ring of polynomials $\mathbb{C}[V]\cong\mathbb{C}[x_1,x_2,\ldots,x_n]$, obtained by treating each basis vector as an indeterminant. If $f\in\mathbb{C}[V]$, then we can define an action of $S_n$ on $\mathbb{C}[V]$ by
\[(\sigma f)(x)=f(\sigma x).  \]
This action turns $\mathbb{C}[V]$ into an infinite dimensional representation of $S_n$. Since every polynomial in $\mathbb{C}[V]$ is a finite sum of homogeneous monomials, as a vector space, $\mathbb{C}[V]$ admits the gradation
\[ \mathbb{C}[V]=\bigoplus_{d\geq 0}\mathbb{C}[V]_d, \]
where $\mathbb{C}[V]_d=\{f\in\mathbb{C}[V]:\text{$f$ is a homogeneous polynomial of degree $d$}\}$.
Thus we say that $\mathbb{C}[V]$ is a \emph{graded representation} of $S_n$. Note that each graded component is a finite dimensional $S_n$-representation, with basis given by the collection of all monomials in  $x_1,\ldots,x_n$ of total degree $d$.\\

A polynomial $f\in\mathbb{C}[V]$ is called \emph{symmetric} if 
\[\sigma f=f  \]
for all $\sigma\in S_n$. The collection of all symmetric polynomials in $\mathbb{C}[V]$ forms a subring of $\mathbb{C}[V]$ called the \emph{ring of symmetric polynomials} in $n$-variables, which we denote by $\Lambda_n$. That is,
\[ \Lambda_n=\{f\in\mathbb{C}[V]:\sigma f=f\text{ for all }\sigma\in S_n\}. \]

The symmetric polynomials are generated (as a $\mathbb{C}$-algebra) by the \emph{power sum symetric polynomials} $p_1,\ldots,p_n$, where
\[p_i=x_1^i+x_2^i+\cdots+x_n^i.  \]
In other words, 
\[\Lambda_n\cong\mathbb{C}[p_1,\ldots,p_n].\]
 
The ring of \emph{coinvariants} of $S_n$ is the quotient ring
\[ \mathbb{C}[V]_{S_n}=\mathbb{C}[V]/(p_1,\ldots,p_n) \]
 by the ideal generated by the symmetric polynomials with no constant term. The ring of symmetric polynomials and the coinvariant ring are also graded representations of $S_n$, and are similarly graded by homogeneous degree:
 \begin{align*}
     \Lambda_n&=\bigoplus_{d\geq 0}\Lambda_n^d,\text{ and}\\
     \mathbb{C}[V]_{S_n}&=\bigoplus_{d\geq 0}\mathbb{C}[V]_{S_n}^d,
 \end{align*}
 where $\Lambda_n^d=\Lambda_n\cap \mathbb{C}[V]_d$, and $\mathbb{C}[V]_{S_n}^d=\mathbb{C}[V]_{S_n}\cap\mathbb{C}[V]^d$. The graded components of these representations are finite dimensional representations of $S_n$, and therefore decompose into finite direct sums of irreducible representations. What makes the coinvariant ring of particular interest, is that it is a graded representation that is isomorphic to the regular representation of $S_n$ \cite{Stanley}. This fact will be important when defining the Gini index of an irreducible representation of $S_n$.
 
 An interesting question one might ask at this point is, ``how do the graded components of the coinvariant ring decompose into sums of irreducible representations?'' To address this question, we recall that the irreducible representations of the symmetric group $S_n$ are indexed by the partitions of $n$, and unlike most other finite groups, there is a canonical way to index them using Specht Modules (as seen in Chapter 2).
Under this identification, the partition $(n)$ indexes the trivial representation and $(1^n)$ indexes the sign representation. 

Let $\lambda$ be a partition of $n$, and let $S^{\lambda}$ be the corresponding irreducible representation (Specht module) indexed by $\lambda$. Denote by $[S^{\lambda}:\mathbb{C}[V]_{S_n}^d]$ the multiplicity of $S^{\lambda}$ in the homogeneous degree $d$ coinvariants of $S_n$. The \emph{graded multiplicity} polynomial of $S^{\lambda}$ in the coinvariant ring is defined by
\[p_{\lambda}(t)=\sum_{d\geq 0} [S^{\lambda}:\mathbb{C}[V]_{S_n}^d]t^d. \]
 
It has been known for some time (a result likely due to Frobenius) that the graded multiplicities of the symmetric group are exactly the Kostka-Foulkes polynomials (\cite{Hall-Littlewood}).

\begin{theorem}
Let $\lambda$ be a partition of $n$ indexing a irreducible representation $S^{\lambda}$ of $S_n$ (in the usual way). Then 
\[ p_{\lambda}(t)=K_{\widetilde{\lambda},(1^n)}(t) \]
\label{graded multiplicity kostka}
\end{theorem}
By applying Theorem \ref{Lascoux & Schutzenberger}, we obtain the following corollary which frames the ``ordinary'' Gini index, $g$, within the context of the representation theory of the Symmetric group.
\begin{corollary}
Let $\lambda$ be a partition of $n$ indexing a irreducible representation $S^{\lambda}$ of $S_n$. The Gini index of $S^{\lambda}$, $g_{S_n}(S^{\lambda})$, is the degree of the graded multiplicity polynomial $p_{\widetilde{\lambda}}(t)$ of $S^{\widetilde{\lambda}}$. Moreover, $g_{S_n}(S^{\lambda})=g(\lambda)$. \label{Gini Sn}
\end{corollary}

We see here that, for the symmetric group, the degree of the graded multiplicity polynomial $p_{\lambda}$ yields the Gini index of the conjugate partition $\widetilde{\lambda}$. We accommodate for this by asserting that the Gini index of a irreducible representation $S^{\lambda}$ of the Symmetric group $S_n$ is the degree of the graded multiplicity polynomial $p_{\widetilde{\lambda}}$. This adjustment is to force the Gini index of the partition $\lambda$ to equal that of the irreducible representation $S^{\lambda}$. Since their irreducible representations do not have a canonical indexing, we will not encounter these difficulties when defining the Gini index for other complex reflection groups. 

At this point the relationship between the degree of the graded multiplicity polynomial and the Gini index may appear a bit forced. In the next chapter, we will see that this trend grows even stronger when we consider the graded multiplicities of irreducible representations of linear algebraic groups.
 
\section{Examples for $S_n$}
\begin{example}{\textbf{The Gini indices of $S_3$ irreps.}}

The symmetric group $S_3$ acts on $\mathbb{C}^3$ by permuting coordinates, and therefore acts on $\mathbb{C}[x_1,x_2,x_3]$. The symmetric polynomials in $3$ variables, $\Lambda_3$, are generated by the power sum polynomials:
\[\Lambda_3=\mathbb{C}[x_1+x_2+x_3,x_1^2+x_2^2+x_3^2,x_1^3+x_2^3+x_3^3].  \]
The partitions of $n=3$ are
\[ (1,1,1),\,(2,1),\text{ and }(3).  \]
These correspond to the sign representation, standard representation, and trivial representation, respectively. The standard tableaux of shape $\lambda$, for all $\lambda\vdash 3$ are
\begin{align*}
    \text{Standard tableaux of }(1,1,1)&:T_1=\young(1,2,3),\\
    \text{Standard tableaux of }(2,1)&:T_2=\young(12,3),\,T_3=\young(13,2),\text{ and}\\
    \text{Standard tableaux of }(3)&:T_4=\young(321).
\end{align*}
The charge statistics of these tableaux are 0, 2, 1, and 3, respectively. By Theorem \ref{Lascoux & Schutzenberger}, the corresponding  Kostka-Foulkes polynomials are
\begin{align*}
    K_{(1,1,1),(1^3)}(t)&=1,\\
    K_{(2,1),(1^3)}(t)&=t^2+t,\text{ and }\\
    K_{(3),(1^3)}(t)&=t^3.
\end{align*}
By Corollary \ref{Gini Sn} we find that the Gini indices of the irreducible representations of $S_3$ are
\begin{align*}
    g_{S_3}(S^{(1,1,1)})&=\deg(p_{(3)}(t))=\deg(K_{(1^3),(1^3)}(t))=0,\\
    g_{S_3}(S^{(2,1)})&=\deg(p_{(2,1)}(t))=\deg(K_{(2,1),(1^3)}(t))=2,\text{ and}\\
    g_{S_3}(S^{(3)})&=\deg(p_{(1,1,1)}(t))=\deg(K_{(3),(1^3)}(t))=3.
\end{align*}
Comparing these values to the ``ordinary'' Gini indices of
\begin{align*}
    g((1,1,1))&=0,\\
    g((2,1))&=2,\text{ and }\\
    g((3))&=3,
\end{align*}
we find, as desired, that they are equal.
\end{example}

\begin{example}{\textbf{The Gini indices of $S_4$ irreps.}}
The partitions of $n=4$ are
\[(1,1,1,1),\,(2,1,1),\,(2,2)\,(3,1),\text{ and }(4).  \]
These label the irreducible representations of $S_4$ as follows
\begin{align*}
    S^{(1,1,1,1)}&=\text{ The sign representation},\\
    S^{(2,1,1)}&=\text{ The standard representation},\\
    S^{(2,2)}&=\text{ (No name)},\\
    S^{(3,1)}&=\text{ sign $\otimes$ standard, and}\\
    S^{(4)}&=\text{The trivial representation}.
\end{align*}
The standard tableaux of shape $\lambda$ (for all $\lambda\vdash 4$), and their corresponding charge statistics are given in figure \ref{figureS}.
\begin{center}
\begin{figure}[h]
    \begin{align*}
    \lambda=(1,1,1,1)&:\,\young(1,2,3,4)\text{ charge}=0\\
    \lambda=(2,1,1)&:\, \young(12,3,4)\text{ charge}=3,\,\,\,\,\,\,\,\,\,\,\,\young(13,2,4)\text{ charge}=2,\,\,\,\,\,\,\,\,\,\,\young(14,2,3)\text{ charge}=1\\
    \lambda=(2,2)&:\,\young(12,34)\text{ charge}=4,\,\,\,\,\,\,\,\,\,\,\,\young(13,24)\text{ charge}=2\\
    \lambda=(3,1)&:\,\young(123,4)\text{ charge}=5,\,\,\,\,\,\young(124,3)\text{ charge}=4,\,\,\young(134,2)\text{ charge}=3\\
    \lambda=(4)&:\,\young(1234)\text{ charge}=6
\end{align*}
    \caption{Standard tableaux of shapes $\lambda\vdash 4$ and their charge statistics.}
    \label{figureS}
\end{figure}
\end{center}
By Theorem \ref{Lascoux & Schutzenberger}, the corresponding Kostka-Foulkes polynomials are    
\begin{align*}
    K_{(1^4),(1^4)}(t)&=1,\\
    K_{(2,1,1),(1^4)}(t)&=t^3+t^2+t,\\
    K_{(2,2),(1^4)}(t)&=t^4+t^2,\\
    K_{(3,1),(1^4)}(t)&=t^5+t^4+t^3,\text{ and}\\
    K_{(4),(q^4)}(t)&=t^6.
\end{align*}
By Corollary \ref{Gini Sn} we find that the Gini indices of the irreducible representations of $S_4$ are
\begin{align*}
    g_{S_4}(S^{(1^4)})&=\deg(p_{(4)}(t))=\deg(K_{(1^4),(1^4)}(t))=0\\
    g_{S_4}(S^{(2,1,1)})&=\deg(p_{(3,1)}(t))=\deg(K_{(2,1,1),(1^4)}(t))=3\\
    g_{S_4}(S^{(2,2)})&=\deg(p_{(2,2)}(t))=\deg(K_{(2,2),(1^4)}(t))=4\\
    g_{S_4}(S^{(3,1)})&=\deg(p_{(2,1,1)}(t))=K_{(3,1),(1^4)}(t)=5,\text{ and}\\
    g_{S_4}(S^{(4)})&=\deg(p_{(1^4)}(t))=\deg(K_{(4),(q^4)}(t))=6.
\end{align*}
Comparing these values to the ``ordinary'' Gini indices of
\begin{align*}
g((1,1,1,1))&=0,\\
g((2,1,1))&=3,\\
g((2,2))&=4,\\
g((3,1))&=5,\text{ and }\\
g((4))&=6.
\end{align*}

\end{example}
we find, as desired, that they are equal.

\section{Invariants of Complex Reflection Groups}

One might wonder whether the structures in the Section 5.1 occur similarly for other finite groups. This is not the true in general, as the results of Section 5.1 relied heavily on the fact that every polynomial can be uniquely written as a finite sum of symmetric polynomials multiplied by co-invariant polynomials. That is, there is an isomorphism
\[\Lambda_n\otimes \mathbb{C}[V]_{S_n}\longrightarrow\mathbb{C}[x_1,\ldots,x_n].  \]
This behavior can be attributed to the fact that the ring of symmetric polynomials is a finitely generated polynomial ring. Without this additional structure, the invariant ring $\mathbb{C}[V]^G$ is quite difficult to understand. It was shown by Shephard and Todd in \cite{Shephard-Todd} that the finite groups whose invariant ring has such a structure are precisely the so-called ``complex reflection groups''.

A linear transformation $r\in GL_n(\mathbb{C})$ is called a \emph{complex reflection} if $r$ has finite order and $r$ fixes a complex hyperplane (a co-dimension $1$ subspace of $\mathbb{C}^n$) pointwise. A \emph{complex reflection group}, $G$, is a finite subgroup of $GL_n(\mathbb{C})$ that is generated by complex reflections.  

Let $G$ be any finite group acting on a finite dimensional complex vector space $V$. Let $x_1,x_2,\ldots,x_n$ be a basis for $V$, and let $x=(x_1,\ldots,x_n)^t$ be the column vector of basis elements. Just as with the symmetric group, we can extend the action of $G$ on $V$ to an action on $\mathbb{C}[V]=\mathbb{C}[x_1,\ldots,x_n]$ by defining
\[ (gf)(x)=f(gx), \]
for any $g\in G$ and $f\in \mathbb{C}[V]$. Then $\mathbb{C}[V]$ is an infinite dimensional graded representation of $G$, and is graded by homogeneous degree. An element $f\in\mathbb{C}[V]$ is called an \emph{invariant} of $G$ (or a ``$G$-invariant'') if
\[gf=f,   \]
for all $g\in G$. The collection of all invariants of $G$ is a ring, and is called the \emph{invariant ring} of $G$; for this we will use the notation
\[\mathbb{C}[V]^G=\{f\in\mathbb{C}[V]:gf=f\text{ for all }g\in G\}.  \]

The aforementioned result due to Shephard and Todd, and later expanded by Chevalley in \cite{Chevalley}, is now presented.
\begin{theorem}(Shephard-Todd-Chevalley Theorem)
Let $V$ be a finite dimensional complex vector space, and let $G$ be a finite subgroup of $GL(V)$. The following are equivalent:
\begin{enumerate}
    \item $G$ is a complex reflection group.
    \item There are $\ell$ algebraically independent non-constant homogeneous polynomials $f_1,f_2,\ldots,f_{\ell}\in\mathbb{C}[V]$ such that
    \[\mathbb{C}[V]^G=\mathbb{C}[f_1,\ldots,f_{\ell}]\text{, and}  \]
    \[|G|=\deg(f_1)\cdot\deg(f_2)\cdot\cdots\cdot\deg(f_{\ell}).\]
\end{enumerate}
\end{theorem}
Shephard and Todd proved this result by deriving a full classification of such groups. The complex reflection groups consist of 34 exceptional groups, and three infinite families; the symmetric, cyclic, and ``imprimitive'' groups - the last of which contains the dihedral groups, which we will examine in detail later in this chapter.

\section{Graded Multiplicities and the Gini Index}
Let $V$ be a $n$-dimensional complex vector space with fixed basis $x_1,\ldots,x_n$, $G\subseteq GL(V)$ a complex reflection group, and $\mathbb{C}[V]^G=\mathbb{C}[f_1,\ldots,f_{\ell}]$ the invariant ring of $G$. The \emph{coinvariant ring} of $G$ is quotient ring
\[\mathbb{C}[V]_G=\mathbb{C}[V]/(f_1,\ldots,f_{\ell})  \]
by the ideal generated by the invariant ring generators. As was the case with the symmetric group, the invariant ring and coinvariant ring of a complex reflection group $G$ are graded representations of $G$. Furthermore, the coinvariant ring is isomorphic to the regular representation of $G$. Let 
\[ \mathbb{C}[V]_G=\bigoplus_{d\geq 0}\mathbb{C}[V]_G^d \]
be the gradation of the coinvariant ring of $G$ by homogeneous degree $d$, where \[\mathbb{C}[V]_G^d=\{f\in\mathbb{C}[V]_G:f\text{ is homogeneous of degree }d\}. \]

Let $V^{\lambda}$ be an irreducible representation of $G$ indexed by an irreducible character $\lambda$ of $G$. Denote by $[V^{\lambda}:\mathbb{C}[V]_G^d]$ the multiplicity of $V^{\lambda}$ in the homogeneous degree $d$ coinvariants of $G$. The \emph{graded multiplicity polynomial} of $V^{\lambda}$ is defined by
\[p_{\lambda}(t)=\sum_{d\geq 0}[V^{\lambda}:\mathbb{C}[V]_G^d]t^d.  \]
In \cite{Stanley}, Stanley proved that the graded multiplicity of a irreducible representation can be calculated using Molien's Theorem as follows.
\begin{theorem}(Stanley)
Let $f_1,\ldots,f_{\ell}$ be the generators of $\mathbb{C}[V]^G$, where $G\subseteq GL(V)$ is a complex reflection group, and set $d_i=\deg(f_i)$. Let $\lambda$ be an irreducible character of $G$, and let $V^{\lambda}$ be the corresponding irreducible representation. The graded multiplicity of $V^{\lambda}$ is given by
\[p_{\lambda}(t)=\frac{1}{|G|}\prod_{i=1}^{\ell}(1-t^{d_i})\sum_{T\in G}\frac{\overline{\lambda}(T)}{\det(I-tT)}.  \]
\label{stanley}
\end{theorem}

We define the Gini index of an irreducible representation of a complex reflection group as follows. If $G$ is the symmetric group $S_n$, then the Gini index of a irreducible representation $S^{\lambda}$ (indexed by the partition $\lambda\vdash n$) is, as before, the degree of the graded multiplicity polynomial of the ``conjugate'' representation:
\[ g_{S_n}(S^{\lambda})=\deg(p_{\widetilde{\lambda}}(t))=\deg(K_{\lambda,(1^n)}(t)).  \]
If $G$ is any other complex reflection group, we define the Gini index of an irreducible representation $V^{\lambda}$ (indexed by an irreducible character $\lambda$) by 
\[g_{G}(V^{\lambda})=\deg(p_{\lambda}(t))\]
The conjugation performed when calculating the Gini index of a irreducible representation of $S_n$ is necessary, as (unlike most other finite groups) the irreducible representations of $S_n$ have an accepted standard indexing, and we want $g_{S_n}$ to agree with the Gini index defined on the index set $P_n$. No adjustments are necessary for other complex reflection groups, as the indexing of their irreducible representations is not standardized.

Recall that the coinvariant ring $\mathbb{C}[V]_G$ of a complex reflection group $G$ is isomorphic to the regular representation of $G$; that is, the representation of $G$ acting on itself by left translations. The regular representation of a finite group $G$ always decomposes as the direct sum of irreducible representations
\[ \bigoplus_{\lambda\in \widehat{G}}\left(V^{\lambda}\right)^{\dim(V^{\lambda})} ,\]
where each irreducible representation $V^{\lambda}$ of $G$ occurs with multiplicity equal to its dimension. in other words, every irreducible representation of $G$ occurs in the decomposition of the regular representation of $G$, and therefore occurs in the decomposition of the coinvariant ring of $G$. Thus the Gini index $g_G$ is well defined. 

\section{The Gini Index of an Irreducible Representation of the Dihedral Group}
Let $n\geq 3$. The dihedral group $D_{2n}$ of order $2n$ is an example of a complex reflection group, and belongs the infinite family of ``imprimitive'' complex reflection groups. We will describe the irreducible characters of $D_{2n}$ and fix an indexing for these characters. We will then apply Stanley's theorem to determine formulas for the Gini index $g_{D_{2n}}$, based on our choice of index set. 

The dihedral group of order ${2n}$, when viewed as the group of rigid motions of the regular $n$-gon, is generated by a rotation $r$ by $\frac{2\pi}{n}$ radians and a reflection $s$, with relations
\[D_{2n}=\langle s,r:s^2=r^n=1, srs=r^{-1} \rangle,  \]
where $1$ is the group identity. All rotations in $D_{2n}$ are of the form $r^k$, and all reflections are of the form $sr^k$,
for $1\leq k \leq n$. Using these facts, we can write out explicit formulas for the irreducible representations of $D_{2n}$. The dihedral group only has $1$- and $2$-dimensional irreducible representations, but the number of irreps with these dimensions depends on the parity of $n$. 

If $n$ is odd, $D_{2n}$ has two $1$-dimensional irreps, whereas $D_{2n}$ has four $1$-dimensional irreps if $n$ is even. These characters are given in tables \ref{odd 1-d characters}, and \ref{even 1-d characters}.

\begin{center}
    \begin{tabular}{|P{1.5cm}||P{1.5cm}|P{1.5cm}|}
 \hline
         &$sr^k$&$r^k$  \\
         \hline
         $\chi^1$&$1$&$1$ \\
         \hline
         $\chi^2$&$-1$&$1$\\
         \hline
    \end{tabular}
    \captionof{table}{$1$-dimensional characters of $D_{2n}$ when $n$ is odd}\label{odd 1-d characters}
\end{center}

\begin{center}
    \begin{tabular}{|P{1.5cm}||P{1.5cm}|P{1.5cm}|}
    \hline
         &$sr^k$&$r^k$  \\
         \hline
         $\chi^1$&$1$&$1$ \\
         \hline
         $\chi^2$&$-1$&$1$\\
         \hline
         $\chi^3$&$(-1)^k$&$(-1)^k$\\
         \hline
         $\chi^4$&$(-1)^{k+1}$&$(-1)^k$\\
         \hline
    \end{tabular}
    \captionof{table}{$1$-dimensional characters of $D_{2n}$ when $n$ is even}\label{even 1-d characters}
\end{center}
The number of $2$-dimensional irreps of $D_{2n}$ also depends on the parity of $n$. If $n$ is odd then $D_{2n}$ has $\frac{n}{2}-1$ $2$-dimensional irreps, whereas $D_{2n}$ has $\frac{n-1}{2}$ $2$-dimensional irreps if $n$ is even. The formulas for these irreps are given by
\begin{align*}
\rho_j(sr^k)&=\begin{bmatrix} 0&e^{-2\pi i j k/n}\\ e^{2\pi i j k/n}&0 \end{bmatrix},\\ 
\rho_j(r^k)&=\begin{bmatrix}e^{2\pi i j k/n}&0\\0&e^{-2\pi i j k/n}
\end{bmatrix},\\
\end{align*}
where $1\leq k \leq n$ and $0<j<\frac{n}{2}$ indexes the $2$-dimensional irreps. 

\begin{example}
The left regular representation $\rho:D_6\longrightarrow GL_6(\mathbb{C})$ is the representation afforded by the left action of $D_6$ on itself. In terms of matrices, $\rho$ is defined by
\[\rho(s)=\begin{bmatrix}0&1&0&0&0&0\\
        1&0&0&0&0&0\\
        0&0&0&0&1&0\\
        0&0&0&0&0&1\\
        0&0&1&0&0&0\\
        0&0&0&1&0&0
        \end{bmatrix}\text{ and}  \]
        \[\rho(r)=\begin{bmatrix}
        0&0&1&0&0&0\\
        0&0&0&0&0&1\\
        0&0&0&1&0&0\\
        1&0&0&0&0&0\\
        0&1&0&0&0&0\\
        0&0&0&0&1&0
        \end{bmatrix}.  \]
By the above discussion, since $D_6=D_{2\cdot 3}$, and $3$ is odd, $D_6$ has two $1$-dimensional irreps, $\chi^1$ and $\chi^2$, and one $2$-dimensional irrep, $\rho_1$. The left regular representation $\rho$ is $6$-dimensional, and decomposes as
\[\rho=\chi^1\oplus\chi^2\oplus 2\rho_1.  \]
Each irreducible representation appears in the decomposition of $\rho$ with multiplicity equal to its dimension, which is true for the left regular representation of any finite group.
\end{example}

To determine formulas for the graded multiplicity polynomials of $D_{2n}$, we represent $D_{2n}$ as a subgroup of a general linear group $GL_2(\mathbb{C})$ via the (faithful) representation $\rho_1$. Under this identification we see that the reflections and rotations of $D_{2n}$ are respectively given by the matrices
\[ sr^k= \begin{bmatrix} 0&e^{-2\pi i k/n}\\ e^{2\pi i k/n}&0 \end{bmatrix}\text{ and}\]
\[ r^k=\begin{bmatrix}e^{2\pi i k/n}&0\\0&e^{-2\pi i k/n}
\end{bmatrix} .\]
The polynomials
\[ f_1(x_1,x_2)=x_1x_2\text{ and} \]
\[ f_2(x_1,x_2)=x_1^n+x_2^n\]
are algebraically independent, and are also invariants of $D_{2n}$. Moreover, since 
\[ \deg(f_1)\cdot\deg(f_2)=2n, \]
by the Shephard-Todd-Chevalley theorem, these polynomials generate the ring of invariants of $D_{2n}$. That is,
\[\mathbb{C}[x_1,x_2]^{D_{2n}}=\mathbb{C}[f_1,f_2].  \]

We will now apply Stanley's theorem to the dihedral group to acquire formulas for the graded multiplicity polynomials, and in turn, the Gini index $g_{D_{2n}}$. 

Let $n\geq 3$. $D_{2n}$ has at least two $1$-dimensional irreducible representations, $\chi^1$ and $\chi^2$. Since these representations are $1$-dimensional, they are the same as their characters, and we may apply Stanley's theorem directly to find their graded multiplicity polynomials:
\begin{align*}
    p_{\chi^1}(t)&=\frac{(1-t^2)(1-t^n)}{2n}\sum_{T\in D_{2n}}\frac{1}{\det(I-tT)}\\
    &=\frac{(1-t^2)(1-t^n)}{2n}\sum_{k=1}^n\left(\frac{1}{(1-t^2)}+\frac{1}{(1-t\omega^k)(1-t\omega^{-k})}\right)\\
    &=\frac{(1-t^n)}{2}+\frac{1}{2n}\sum_{k=1}^n\frac{(1-t^2)(1-t^n)}{(1-t\omega^k)(1-t\omega^{-k})}\\
    &=\frac{(1-t^n)}{2}+\frac{(1+t^n)}{2}\\
    &=1,
\end{align*}
where $\omega=e^{2\pi i/n}$. Similarly, we find that 
\[p_{\chi^2}(t)=t^n,  \]
and if $n$ is even,
\begin{align*}
    p_{\chi^3}(t)&=t^{n/2}\text{, and}\\
    p_{\chi^4}(t)&=t^{n/2}.
\end{align*}
    To determine the graded multiplicity polynomials of the $2$-dimensional irreps $\rho_j$, we first calculate their characters $\chi_j$:
    \begin{align*}
        \chi_j(sr^k)&=0\text{ and}\\
        \chi_j(r^k)&=2\cos({2\pi jk}/{n}),
    \end{align*}
    for $1\leq k \leq n$ and $0<j<n/2$.
    Applying Stanley's theorem yields
    \[p_{\chi_j}(t)= t^{n-j}+t^j. \]
    The values of the Gini index $g_{D_{2n}}$ on the irreducible representations of the dihedral group $D_{2n}$ are given in table \ref{dihedral gini}.
    
    \begin{center}
    \begin{tabular}{|P{3cm}||P{3cm}|P{3cm}|P{3cm}|}
    \hline
         {Irreducible Representation}&{Dimension}&{Graded Multiplicity}&{Gini Index, $g_{D_{2n}}$}  \\
         \hline\hline
         $\chi^1$&$1$&$1$&$0$ \\
         \hline
         $\chi^2$&$1$&$t^n$&$n$\\
         \hline
         $\chi^3$ (for $n$ even)&$1$&$t^{n/2}$&$\frac{n}{2}$\\
         \hline
         $\chi^4$ (for $n$ even)&$1$&$t^{n/2}$&$\frac{n}{2}$\\
         \hline
         $\rho_j$ ($0<j<n/2$)&$2$&$t^{n-j}+t^j$&${n-j}$\\
         \hline
    \end{tabular}
    \captionof{table}{Values of the Gini Index $g_{D_{2n}}$ on an irreducible representation of $D_{2n}$.}\label{dihedral gini}
\end{center}
    

\newpage

\chapter{The Gini Index and Connected Reductive Linear Algebraic Groups}
Let $n$ and $k$ be positive integers, and let $\alpha$ be a decreasing sequence of $n$ integers. Such sequences index the irreducible rational representations, $V^{\alpha}$, of $GL_n(\mathbb{C})$. As was the case with complex reflection groups, $GL_n(C)$ has certain graded representations of invariants and \emph{harmonics} that exhibit special properties. To the representation $V^{\alpha}$ we associate a polynomial, $p_{\alpha}$, which is again called the ``graded multiplicity'' of $V^{\alpha}$ in the harmonics. To the sequence $\alpha$ there is also an associated pair of partitions $\lambda$ and $\mu$ of $nk$, and it turns out that the Kostka-Foulkes polynomial $K_{\lambda,\mu}$ is the graded multiplicity of $V^{\alpha}$ in the harmonics. We will see that the degrees of the graded multiplicity polynomials $p_{\alpha}$ are exactly the values of the Gini index $g_{nk,n}$ on $P_{nk}$.

The general linear group is the principle example of a connected reductive linear algebraic group. Using graded multiplicities, we will extend the notion of the Gini index to all other connected reductive linear algebraic groups.

\section{Harmonics of Connected Reductive Linear Algebraic Groups}

Let $G$ be a connected reductive linear algebraic group over $\mathbb{C}$. The general linear group belongs to this family, as does any semisimple linear algebraic group. Let $V$ be a $n$-dimensional rational representation of $G$, and choose a basis $x_1,x_2,\ldots,x_n$ for $V$. Let $\mathbb{C}[V]=\mathbb{C}[x_1,x_2,\ldots,x_n]$ denote the algebra of polynomial functions on $V$. This algebra is a infinite dimensional graded representation of $G$ with gradation
\[ \mathbb{C}[V]=\bigoplus_{d\geq 0}\mathbb{C}[V]_d, \]
where $\mathbb{C}[V]_d$ is the vector space of homogeneous degree $d$ polynomials in $\mathbb{C}[V]$. Just as we defined the invariant and coinvariant rings of a complex reflection group, we can define the invariant and coinvariant rings of a linear algebraic group $G$. Let \[\mathbb{C}[V]^G=\{f\in\mathbb{C}[V]:gf=f\text{ for all }g\in G \}\] 
be the ring of $G$-invariant polynomials in $\mathbb{C}[V]$. The coinvariant ring of $G$ is the quotient

\[ \mathbb{C}[V]_G=\mathbb{C}[V]/\mathbb{C}[V]_+^G, \]
of the full polynomial ring by the ideal $\mathbb{C}[V]^G_+$ of $G$-invariant polynomials without constant term. When discussing the coinvariants of a linear algebraic group it is more common to define them in terms of $G$-harmonic functions.\\

Let $\partial_i=\frac{\partial}{\partial x_i}$ and, for $f\in\mathbb{C}[V]$, define 
\[f(\partial)=f(\partial_1,\partial_2,\ldots,\partial_n).  \]
The constant coefficient differential operators on $\mathbb{C}[V]$ is the set
\[\mathcal{D}(V)=\{f(\partial):f\in\mathbb{C}[V]  \}.  \]
Denote by $\mathcal{D}[V]_+$ the set of differential operators without constant term, and by $\mathcal{D}[V]^G_+$ the set of $G$-invariant differential operators with no constant term. The module of $G$-\emph{harmonic} polynomials defined as
\[ \mathcal{H}(V)=\{f\in\mathbb{C}[V]:\Delta f=0\text{ for all }\Delta\in\mathcal{D}(V)^G_+\}. \]
The $G$-harmonic functions, and $G$-coinvariants are isomorphic as graded representations of $G$, and are graded by homogeneous degree:
\[ \mathcal{H}(V)=\bigoplus_{d\geq 0}\mathcal{H}^d(V), \]
where $\mathcal{H}^d(V)=\mathcal{H}(V)\cap\mathbb{C}[V]_d$. This fact is non-trivial, since it says that if a function is harmonic, then so are its homogeneous components. In general, every polynomial function can be expressed as a sum of $G$-invariant functions multiplied by $G$-harmonic functions. In other words, there is a surjection
\[ \mathbb{C}[V]^G\otimes\mathcal{H}(V)\longrightarrow\mathbb{C}[V]\longrightarrow 0 \]
obtained by linearly extending multiplication. This corresponds to the separation of variables that one does when studying the Laplace operator in two dimensions. The two tensor components correspond to the decomposition into radial and spherical parts, respectively. A guiding question in the study of $G$-harmonic functions was (and is), ``When is the above map an isomorphism?" That is, when is the sum of products of invariants and harmonics unique? Equivalently, when is $\mathbb{C}[V]$ a free module over $\mathbb{C}[V]^G$? This question was partially answered by Kostant in his pivotal 1963 paper \emph{Lie group representations on polynomials rings} (see \cite{Kostant}). \\

Let $\mathfrak{g}$ denote the Lie algebra of $G$, and let $\text{Ad}:G\longrightarrow GL(\mathfrak{g})$ denote the Adjoint representation of $G$, defined by
\[ (\text{Ad}x)(A)=xAx^{-1}, \]
for $x\in G$ and $A\in\mathfrak{g}$. Among the many results in \cite{Kostant}, Kostant proved that $\mathbb{C}[\mathfrak{g}]$ is a free module over $\mathbb{C}[\mathfrak{g}]^G$, for any connected reductive group $G$.

\section{Graded Multiplicities and the Gini Index}

Let $G$ be a connected reductive linear algebraic group, and let $\mathfrak{g}$ be the Lie algebra of $G$. Let $\mathcal{H}(\mathfrak{g})$ be the infinite-dimensional module of $G$-harmonic functions. The graded components of $\mathcal{H}(\mathfrak{g})$ are all finite dimensional, and since $G$ is reductive, $\mathcal{H}^d(\mathfrak{g})$ is completely reducible, for each $d\geq 0$. Given a dominant weight $\alpha\in P_{++}(G)$ of $G$, what is the multiplicity,
\[[V^{\alpha}:\mathcal{H}^d(\mathfrak{g})]=?,\]
of $V^{\alpha}$ in $\mathcal{H}^d(\mathfrak{g})$, where $V^{\alpha}$ is the irreducible rational representation of $G$ with highest weight $\alpha$. A natural thing to do, as in Chapter 5, is consider the, indeed, polynomial defined by the series
\[ p_{\alpha}(t)=\sum_{d\geq 0}[V^{\alpha}:\mathcal{H}^d(\mathfrak{g})]t^d, \]
which is called the \emph{graded multiplicity} of $V^{\alpha}$ in the $G$-harmonic functions. These polynomials extract deep information in representation theory, but outside of Kostant's setting, very little is known about them. Hesselink showed in \cite{Hesselink} that if $G$ is semisimple, then $p_{\alpha}(t)$ can be expressed as an alternating sum in terms of Kostant's partition function. In particular, if $\mathfrak{g}$ is of Lie type $A$, then Hesselink's alternating sum formula shows that $p_{\alpha}(t)$ is equal to a Kostka-Foulkes polynomial $K_{\lambda,\mu}(t)$, where $\lambda$ and $\mu$ are integer partitions that depend (to a certain extent) on $\alpha$. This case is explored in the next section.

Kostant defined the \emph{generalized exponents} of $V^{\alpha}$ to be the exponents $e_1,e_2,\ldots,e_s$ of the nonzero terms in the graded multiplicity,
\[p_{\alpha}(t)=\sum_{d\geq 0}[V^{\alpha}:\mathcal{H}^d(\mathfrak{g})]t^d=\sum_{i=1}^s c_it^{e_i},  \]
of $V^{\alpha}$ in the $G$-harmonic functions. Due to the connections between $p_{\alpha}(t)$ and the Kostka-Foulkes polynomials in Lie type A, we define the Gini index $g_G(V^{\alpha})$ of an irreducible rational representation of $G$ as the maximum of the generalized exponents,
\[ g_G(V^{\alpha})=\max\{e_1,\ldots,e_s\}, \]
or equivalently, as the degree of the graded multiplicity of $V^{\alpha}$ in the $G$-harmonic functions,
\[ g_G(V^{\alpha})=\deg(p_{\alpha}(t)).  \]
Unlike the case of complex reflection groups, not every irreducible rational representation of $G$ occurs in the decomposition of the harmonics. So that $g_G$ is well defined, we will adopt the convention that the degree of the zero polynomial is $-\infty$. Thus if $V^{\alpha}$ is an irreducible rational representation of $G$, and $p_{\alpha}(t)=0$, the Gini index is
\[g_G(V^{\alpha})=-\infty.  \]

We will see in the following section that if $G=GL_n(\mathbb{C})$ and $V^{\alpha}$ is an irreducible polynomial representation of $G$, then the Gini index $g_{GL_n(\mathbb{C})}(V^{\alpha})$ is given by the Gini index $g_{nk,n}(\lambda)$, for a certain partition $\lambda$ of $nk$.

\newpage
\section{The Gini index of an irreducible representation of $GL_n(\mathbb{C})$}
As seen in Chapter 2, the highest weights of the irreducible rational representations of  $GL_n(\mathbb{C})$ are canonically indexed by the set of $n$-tuples of non-increasing integers. Let $\alpha=(\alpha_1,\alpha_2,\ldots,\alpha_n)$ and $\beta=(\beta_1,\beta_2,\ldots,\beta_n)$ be two dominant weights of $GL_n(\mathbb{C})$. The sum $\alpha+\beta$ is defined as a sum of vectors
\[ \alpha+\beta=(\alpha_1+\beta_1,\alpha_2+\beta_2,\ldots,\alpha_n+\beta_n). \]
The dominant weights of $GL_n(\mathbb{C})$ are equipped with a equivalence relation, under which equivalent weights yield isomorphic highest weight representations. Two dominant weights, $\alpha$ and $\beta$, are equivalent if there is a positive integer $k$ such that
\[\alpha=\beta+(k^n),  \]
where $(k^n)=(k,k,\ldots,k)$ is the so-called ``flat'' partition of $nk$ containing $n$ values of $k$. 

Let $\alpha=(\alpha_1,\alpha_2,\ldots,\alpha_n)$ be a non-zero dominant weight of $GL_n(\mathbb{C})$ such that
\[\alpha_1+\alpha_2+\cdots+\alpha_n=0.  \]
As $\alpha$ is non-zero, the last term, $\alpha_n$, in $\alpha$ is always negative. Let $k$ be an integer with $k\geq-\alpha_n$, and define
\[\lambda=\alpha+(k^n). \]
Then $\lambda$ is a partition of $nk$ with at most $n$ parts. If $k=-\alpha_n$, then $\lambda_n=0$ so $\lambda$ is a partition of $nk$ with at most $n-1$ parts. The weights $\lambda$ and $\alpha$ are both dominant weights of $GL_n(\mathbb{C})$, and are equivalent; hence the corresponding irreducible highest weight representations $V^{\lambda}$ and $V^{\alpha}$ are isomorphic. \\

Let $\alpha$ be a dominant weight of $GL_n(\mathbb{C})$. When does the irreducible representation $V^{\alpha}$ of $GL_n(\mathbb{C})$ of highest weight $\alpha$ appear in the graded decomposition of the harmonics of $GL_n(\mathbb{C})$? That is, when is the graded multiplicity polynomial $p_{V^{\alpha}}(t)$ nonzero? The answer, due to Kostant, is that the only irreducible highest weight representations of $GL_n(\mathbb{C})$ with positive multiplicity in $\mathcal{H}(\mathfrak{g})$ are those with highest weights that sum to zero (c.f. \cite{Kostant}). Combining this fact with the details of the above discussion yields the following theorem (originally shown in \cite{Hall-Littlewood}):

\begin{theorem}
Let $V^{\alpha}$ be an irreducible representation of $GL_n(\mathbb{C})$ with highest weight $\alpha\in\mathbb{Z}^n$ such that $\alpha_1\geq\cdots\geq\alpha_n$. 
\begin{enumerate}
    \item If $\alpha_1+\ldots+\alpha_n\neq 0$, then $p_{V^{\alpha}}(t)=0$. 
    \item If $\alpha_1+\ldots+\alpha_n=0$, then 
    \[p_{V^{\alpha}}(t)=K_{\lambda,(k^n)}(t),  \]
    where $k\geq |\alpha_n|$, and $\lambda=\alpha+(k^n)$.
\end{enumerate}
\end{theorem}
We then obtain the following corrolary, which characterizes the Gini index $g_{GL_n(\mathbb{C})}$ on the irreducible representations of $GL_n(\mathbb{C})$. 
\begin{corollary}
Let $V^{\alpha}$ be a irreducible representation of $GL_n(\mathbb{C})$ with highest weight $\alpha\in\mathbb{Z}^n$ such that $\alpha_1\geq\cdots\geq\alpha_n$. 
\begin{itemize}
    \item If $\alpha_1+\cdots+\alpha_n\neq 0$, then \[g_{GL_n(\mathbb{C})}(V^{\alpha})=-\infty.\]
    \item If $\alpha_1+\cdots+\alpha_n=0$, then
    \[ g_{GL_n(\mathbb{C})}(V^{\alpha})=\deg(K_{\lambda,(k^n)}(t)), \]
    where $k\geq |\alpha_n|$ and $\lambda=\alpha+(k^n)$.
\end{itemize}
\label{GLn Gini Characterization}
\end{corollary}
By Theorem \ref{Lascoux & Schutzenberger}, we see that if $\alpha_1+\cdots+\alpha_n=0$, then 
\begin{align*}
    g_{GL_n(\mathbb{C})}(V^{\alpha})&=b((k^n))-b(\lambda)\\
    &=g_{nk,n}(\lambda).
\end{align*}
In other words, if we disregard the irreducible representations whose highest weight $\alpha$ has nonzero sum, the Gini index for $GL_n(\mathbb{C})$ is precisely the Gini index $g_{nk,n}$ defined on the set of partitions of $nk$ with at most $n$ parts. Since the choice of $k\geq|\alpha_n|$ in theorem \ref{GLn Gini Characterization} was immaterial, we obtain the following Corollary on the Gini index on partitions of $nk$ with at most $n$ parts:
\begin{corollary}
Let $n$, and $k$ be positive integers, and let $\lambda$ be a partition of $nk$ with at most $n$ parts. Let $j\geq k$, and define
\[ \mu=\lambda-(k^n)+(j^n). \]
Then $\mu$ is a partition of $nj$ with at most $n$ parts, and
\[g_{nk,n}(\lambda)=g_{nj,n}(\mu).  \]
\label{g nk nj corollary}
\end{corollary}
\begin{example}
To illustrate Corollary \ref{g nk nj corollary}, suppose we have two populations of $n=5$ people. Amongst the first population is distributed $nk=15$ dollars so that two people have $5$ dollars each, two have $2$ dollars each, and one person has $1$ dollar. That is, the first distribution corresponds to the partition
\[ \lambda=(5,5,2,2,1). \]
Amongst the second population suppose there is distributed $nj=25$ dollars in a fashion corresponding to the partition
\[ \mu=(7,7,4,4,3) \]
Then the Gini indices of these distributions are
\begin{align*}
    g_{15,5}(\lambda)&=b((3^5))-b(\lambda)=11,\text{ and}\\
    g_{25,5}(\mu)&=b((5^5))-b(\mu)=11.\\
\end{align*}
The reasons for this equality become obvious when we look at the graphs of the Lorenz curves of $\lambda$ and $\mu$, which are given in figures \ref{Lorenz 1} and \ref{Lorenz 2}, respectively.

The Lorenz curve of $\mu$ is simply that of $\lambda$ translated up $2$-units. But that translation increases the amount of money in circulation, and results in the line of equality also being translated up $2$-units --- leaving the Gini index invariant.

\begin{figure}[htp]
    \centering
    \includegraphics[width=10cm]{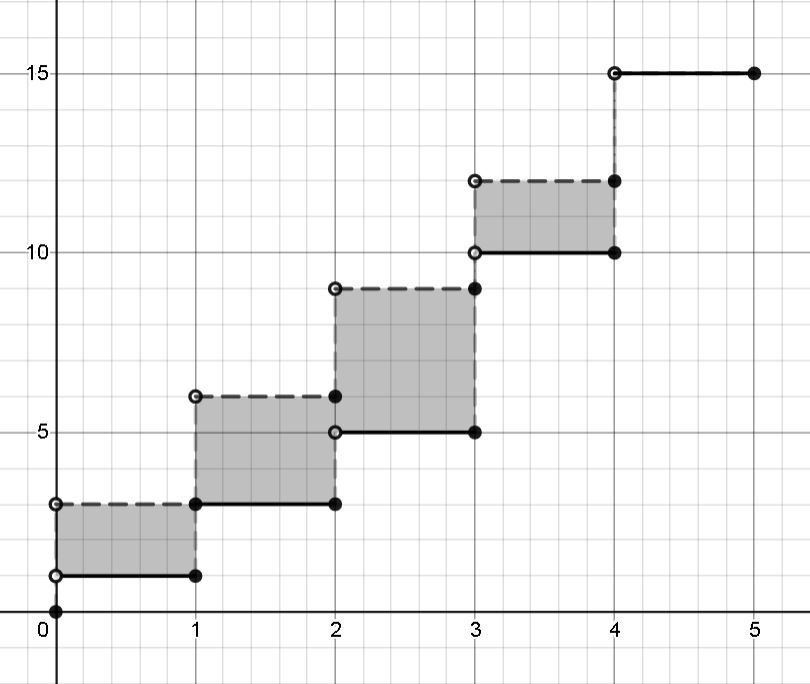}
    \caption{The line of equality (dashed), the Lorenz curve of the partition (5,5,2,2,1) of 15 (solid), and the area between them (shaded).}
    \label{Lorenz 1}
\end{figure}

\begin{figure}[htp]
    \centering
    \includegraphics[width=10cm]{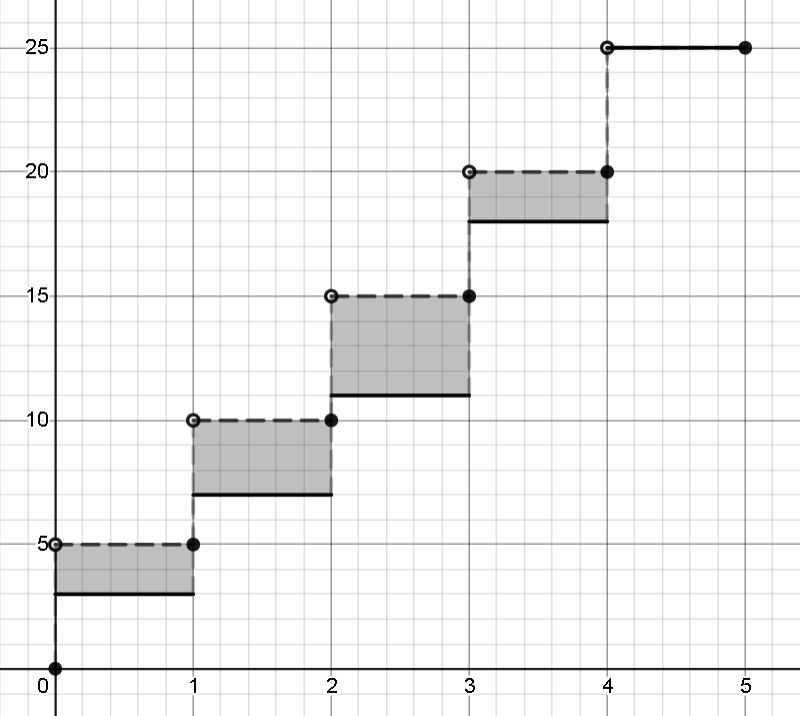}
    \caption{The line of equality (dashed), the Lorenz curve of the partition (7,7,4,4,3) of 25 (solid), and the area between them (shaded).}
    \label{Lorenz 2}
\end{figure}
\end{example}

\newpage
\section{Examples for $GL_n(\mathbb{C})$}
\begin{example}{\textbf{Gini index of $V^{(2,0,-2)}$ on $GL_3(\mathbb{C})$}}
Suppose there are $6$ dollars distributed amongst $3$ people, where one person has $4$ dollars, one has $2$, and the last person has $0$ dollars. This distribution corresponds to the parition $\lambda=(4,2,0)$ of $6$. Letting $n=3$ (the number of people in the population), and $k=2$ (the number of dollars per person in the most equitable distribution), we can find the Gini index $g_{nk,n}$ of $\lambda$, as in Chapter 3, by calculating
\begin{align*}
    g_{6,3}((4,2,0))&=b((2,2,2))-b((4,2,0))\\
    &=(0+2+4)-(0+2+0)\\
    &=6-2\\
    &=4.
\end{align*}
The distribution $\lambda=(4,2,0)$ corresponds to the $GL_3(\mathbb{C})$-dominant weight
\begin{align*}
    \alpha&=\lambda-(2^3)\\
    &=(4,2,0)-(2,2,2)\\
    &=(2,0,-2).
\end{align*}
The graded multiplicity of the highest weight representation $V^{\alpha}$ of $GL_3(\mathbb{C})$ is given by the Kostka-Foulkes polynomial
\[p_{V^{\alpha}}(t)=K_{\lambda,(2^3)}(t),  \]
which we will calculate using Theorem 18.

There are $3$ semi-standard Young tableaux with shape $\lambda=(4,2)$ and weight $(2,2,2)$. These tableaux and their charges are given in Figure \ref{GLn3 example}

\begin{figure}[h]
\begin{center}
    $\young(1122,33)\text{, charge }=4$\\
    $\young(1133,22)\text{, charge }=3$\\
    $\young(1123,23)\text{, charge }=2$
\end{center}    
\caption{Semi-standard Young tableaux with shape $(4,2)$ and weight $(2^3)$, and their corresponding charges.}
\label{GLn3 example}
\end{figure}
Thus the Kostka-Foulkes polynomial (and graded multiplicity of $V^{\alpha}$) is
\[ p_{V^{\alpha}}(t)=K_{\lambda,(2^3)}(t)=t^4+t^3+t^2. \]
Therefore, the Gini index of $V^{\alpha}$ is
\[ g_{GL_3(\mathbb{C})}(V^{\alpha})=\deg(p_{V^{\alpha}}(t))=4, \]
and we see that  $g_{6,3}(\lambda)=g_{GL_3(\mathbb{C})}(V^{\alpha})$.
\end{example}

\begin{example}{\textbf{Gini index of $V^{(2,1,0,-3)}$ on $GL_4(\mathbb{C})$.}}
Suppose there are 4 people in a population, amongst whom is distributed 12 dollars, where one person has 5 dollars, one has 4, one has 3 dollars, and the last has 0. This distribution corresponds to the partition $\lambda=(5,4,3,0)$ of $12$. Letting $n=4$ (the number of people in the population), and $k=3$ (the number of dollars per person in the most equitable distribution), we can find the Gini index $g_{nk,n}$ of $\lambda$, as in Chapter 3, by calculating
\begin{align*}
    g_{12,4}((5,4,3,0))&=b((3,3,3,3))-b((5,4,3,0))\\
    &=(0+3+6+9)-(0+4+6+0)\\
    &=18-8\\
    &=8.
\end{align*}
The distribution $\lambda=(4,4,3,1)$ corresponds to the $GL_4(\mathbb{C})$-dominant weight
\begin{align*}
    \alpha&=\lambda-(3^4)\\
    &=(5,4,3,0)-(3,3,3,3)\\
    &=(2,1,0,-3).
\end{align*}

The graded multiplicity of the highest weight representation $V^{\alpha}$ is given by the Kostka-Foulkes polynomial
\[p_{V^{\alpha}}(t)=K_{\lambda,(3^4)}(t),  \]
which we will calculate using Therem \ref{Lascoux & Schutzenberger}.

There are $8$ semi-standard Young tableaux with shape $\lambda=(5,4,3,0)$ and weight $(3,3,3,3)$. These tableaux and their charges are given in Figure \ref{GLn example}, and show that the Kostka-Foulkes polynomial (and graded multiplicity of $V^{\alpha}$) is
\[ p_{V^{\alpha}}(t)=K_{\lambda,(3^4)}(t)=t^4+2t^5+2t^6+2t^7+t^8. \]
\begin{center}
\begin{figure}[h]
\begin{align*}
    \young(11123,2333,444),\,\text{charge}=8\hspace{2cm}&\young(11123,2233,444),\,\text{charge}=7\\
    \young(11123,2234,344),\,\text{charge}=6\hspace{2cm}&\young(11124,2233,344),\,\text{charge}=6\\
    \young(11122,2334,344),\,\text{charge}=7\hspace{2cm}&\young(11134,2224,334),\,\text{charge}=4\\
    \young(11124,2234,334),\,\text{charge}=5\hspace{2cm}&\young(11133,2224,344),\,\text{charge}=5
\end{align*}
\caption{Semi-standard Young tableaux with shape $(5,4,3)$ and weight $(3^4)$, and their corresponding charges.}
\label{GLn example}
\end{figure}
\end{center}
Therefore, the Gini index of $V^{\alpha}$ is
\[g_{GL_4(\mathbb{C})}(V^{\alpha})=\deg(p_{V^{\alpha}}(t))=8,  \]
and we see that $g_{12,4}(\lambda)=g_{GL_4(\mathbb{C})}(V^{\alpha})$.
\end{example}

\begin{example}{\textbf{Gini index of $V^{(3,2,-1,-4)}$ on $GL_4(\mathbb{C})$}}
Suppose there are $4$ people in a population, amongst whom is distributed $16$ dollars according to the partition $\lambda=(7,6,3,0)$. Letting $n=4$ (the number of people in the population), and $k=4$ (the number of dollars per person in the most equitable distribution), we can find the Gini index $g_{nk,n}$ of $\lambda$, as in Chapter 3, by calculating,
\begin{align*}
    g_{16,4}((7,6,3,0))&=b((4,4,4,4))-b((7,6,3,0))\\
    &=(0+4+8+12)-(0+6+6+0)\\
    &=24-12\\
    &=12.
\end{align*}
The distribution $\lambda=(7,6,3,0)$ corresponds to the $GL_4(\mathbb{C})$-dominant weight
\begin{align*}
    \alpha&=\lambda-(4^4)\\
    &=(7,6,3,0)-(4,4,4,4)\\
    &=(3,2,-1,-4).
\end{align*}
The graded multiplicity of the highest weight representation $V^{\alpha}$ of $GL_4(\mathbb{C})$ is given by the Kostka-Foulkes polynomial
\[ p_{V^{\alpha}}(t)=K_{\lambda,(4^4)}(t), \]
which we will calculate using Theorem \ref{Lascoux & Schutzenberger}.

There are $16$ semi-standard Young tableaux with shape $\lambda=(7,6,3)$ and weight $(4,4,4,4)$. These tableaux and their charges are given in Figure \ref{GLn example 4}, and show that the Kostka-Foulkes polynomial (and graded multiplicity of $V^{\alpha}$) is
\[ p_{V^{\alpha}}(t)=K_{\lambda,(4^4)}(t)=t^{12}+2t^{11}+3t^{10}+4t^{9}+3t^{8}+2t^{7}+t^{6}. \]
\begin{center}
\begin{figure}[h]
\begin{align*}
    \young(1111234,222444,333)\text{ , charge }=8 \hspace{2cm}&\young(1111334,222244,334)\text{ , charge }=6\\
    \young(1111224,223344,334)\text{ , charge }=9 \hspace{2cm}&\young(1111223,223444,334)\text{ , charge }=10\\
    \young(1111233,222444,334)\text{ , charge }=9\hspace{2cm}&\young(1111234,222344,334)\text{ , charge }=7\\
    \young(1111333,222244,344)\text{ , charge }=7\hspace{2cm}&\young(1111222,233334,444)\text{ , charge }=12\\
    \young(1111234,222334,344)\text{ , charge }=8\hspace{2cm}&\young(1111233,222344,344)\text{ , charge }=8\\
    \young(1111222,233344,344)\text{, charge }=11\hspace{2cm}&\young(1111224,223334,344)\text{ , charge }=10\\
    \young(1111223,223344,344)\text{ , charge }=9\hspace{2cm}&\young(1111233,222334,444)\text{ , charge }=9\\
    \young(1111224,223333,444)\text{, charge }=11\hspace{2cm}&\young(1111223,223334,444)\text{ , charge }=10
\end{align*}
\caption{Semi-standard Young tableaux with shape $(7,6,3)$ and weight $(4^4)$, and their corresponding charges.}
\label{GLn example 4}
\end{figure}
\end{center}

Therefore, the Gini index of $V^{\alpha}$ is
\[ g_{GL_4(\mathbb{C})}(V^{\alpha})=\deg(p_{V^{\alpha}}(t))=12, \]
and we see that $g_{16,4}(\lambda)=g_{GL_4(\mathbb{C})}(V^{\alpha})$.
\end{example}
In each of these examples, we have computed the Gini index of a partition $\lambda\vdash nk$ for which $\lambda_n=0$. Using Corollary \ref{g nk nj corollary}, these Gini indices are equal to those of any $\mu=\lambda+(j^n)$, for any $j\in\mathbb{N}$. The Kostka-Foulkes polynomials $K_{\lambda,(k^n)}(t)$ and $K_{\mu,(j^n)}(t)$ will be equal, but the semi-standard Young tableaux of shape $\mu$ and weight $(j^n)$ are more difficult to work with than those of shape $\lambda$ and weight $(k^n)$. In other words, we intentionally looked at the simplest distributions, those for which $\lambda_n=0$, in order to simplify our computations. The results, however, would have been similar had we chosen more complicated examples.

\newpage
\addcontentsline{toc}{chapter}{Curriculum Vitae}
\chapter*{Curriculum Vitae}
\begin{center}
    \textbf{Grant Kopitzke}
\end{center}

\begin{rSection}{Education}
\begin{itemize}
    \item \textbf{University of Wisconsin, Milwaukee}\hfill{May 2021}\\Ph.D. in Mathematics\\Dissertation Title: \textit{``The Gini Index in Algebraic Combinatorics and Representation Theory''}\\Advisor: Dr. Jeb Willenbring

    \item \textbf{University of Wisconsin, Milwaukee}\hfill{May 2019}\\{M.S. in Mathematics}
   
    \item \textbf{University of Wisconsin, Oshkosh}\hfill{May 2017}\\B.S. in Mathematics\\Thesis Title: \textit{``Congruences of the 11 and 13-Regular Partition Function''}
    
    \item \textbf{University of Wisconsin, Fox Valley}\hfill{May 2014}\\ A.A.S. in Mathematics
\end{itemize}
\end{rSection}

\begin{rSection}{Teaching Experience}
\textbf{University of Wisconsin, Milwaukee}\hfill{2017--2021}\\{Teaching Assistant -- Had full responsibility for the preparation, instruction, and grading of the following courses.}

\begin{itemize}
    \item \textbf{Math 431 -- Modern Algebra with Applications}\hfill Fall 2020\\Responsibility for course design, syllabus development, instruction and grading of one section. Wrote all new course notes in book form -- which will be revised and submitted for future publication. Course topics included groups, rings, fields, Boolean algebras, Diffie-Helman key exchanges, RSA cryptosystems, El Gamal cryptosystems, elliptic curve cryptography, cryptanalysis, and circuit simplification.
    \item \textbf{Math 276 -- Alg. Structures for Elementary Ed. Majors}\hfill Spring 2020, 2021\\Responsible for creation of all course materials, including syllabi, homework assignments, quizzes, and exams. Course topics included mathematical logic, sets, functions, groups, rings, and elementary number theory.
    \item \textbf{Math 105 -- Intro to College Algebra Online}\hfill Summer 2020\\Responsible for creation of all course materials, including syllabi, online homework system, quizzes, and exams. Course topics included quadratic equations, rational expressions, exponential and logarithmic functions, and rational exponents.
    \item \textbf{Math 231 -- Calculus and Analytic Geometry 1}\hfill Summer 2019\\Responsible for creation of all course materials, including syllabi, homework assignments, quizzes, and exams. Course topics included limits, derivatives, graphing, antiderivatives, integrals, and applications.
    \item \textbf{Math 103 -- Contemporary Applications of Mathematics}\hfill Spring 2019\\Responsible for instruction and grading of two sections. Course topics included voting theory, fair division, apportionment, graph theory, financial mathematics, and statistical inference.
    \item \textbf{Math 211 -- Survey in Calculus and Analytic Geometry}\hfill Fall 2018\\Responsible for instruction and grading of three discussion sections. Course topics included coordinate systems, equations of curves, limits, differentiation, integration, and applications.
    \item \textbf{Math 105 -- Intro to College Algebra}\hfill Fall 2017, Spring 2018, Fall 2019\\Responsible for instruction and grading of six sections over three semesters. Course topics included quadratic equations, rational expressions, exponential and logarithmic functions, and rational exponents.
\end{itemize}
\end{rSection}

\begin{rSection}{Research Interests}
\begin{itemize}
\item Algebraic Combinatorics
\item Representation Theory
\item Algebraic Groups
\item Enumerative Combinatorics
\item Number Theory
\item Cryptology
\item Mathematics Education
\end{itemize}
\end{rSection}

\begin{rSection}{Publications and Papers}
\begin{itemize}\label{}
\item \textit{\textbf{``The Gini Index in Alg. Combinatorics and Rep. Theory''}}\hfill 2021\\Ph.D. Dissertation -- University of Wisconsin, Milwaukee

\item \textit{\textbf{``The Gini Index of an Inter Partition''}}\hfill 2020\\Journal of Integer Sequences

\item \textit{\textbf{``Self-Similarity of the 11-Regular Partition Function''}}\hfill 2017\\Oshkosh Scholar (Undergraduate Paper)
\end{itemize}
\end{rSection}

\begin{rSection}{Related Experience}
\begin{itemize}
\item \textbf{University of Wisconsin, Milwaukee}\hfill{Spring 2021}\\
\textbf{Senior Project Advisor:} Co-advised an undergraduate student's senior thesis.\\
\textbf{Topic:} Applications of the $Ax+B$ group in cryptography.
\item \textbf{University of Wisconsin, Milwaukee}\hfill{Spring 2021}\\
\textbf{Grader:} Graded assignments for ``Highschool Mathematics from an Advanced Viewpoint."
\item \textbf{University of Wisconsin, Milwaukee}\hfill{Spring 2021}\\\textbf{Placement Test Monitor: } Proctored online mathematics placement tests.
\item \textbf{University of Wisconsin, Milwaukee}\hfill{Fall 2020}\\{\textbf{Teaching Assistant for Mathematical Literacy for College Students:} Assisted the primary instructor with grading, proctoring, etc..}
\item \textbf{University of Wisconsin, Milwaukee}\hfill{Spring 2020}\\\textbf{Calculus Testing Center Staff: } Administered, proctored and graded proficiency tests for Calculus I and II.
\item \textbf{University of Wisconsin, Oshkosh \& Fox Valley}\hfill{2013-2017}\\\textbf{Mathematics Tutor: }Tutored students in the walk-in math lab.
\end{itemize}
\end{rSection}

\begin{rSection}{Talks}
\begin{itemize}
\item \textbf{Binghampton Uni. Grad. Conference in Alg. and Top.} \hfill Fall 2020\\
\textit{``the Gini Index and Representations of the Symmetric Group''}
\item \textbf{Algebra Seminar (UW Milwaukee)}\hfill Spring 2020, Fall 2020
\begin{itemize}
    \item \textit{The Gini index of an Integer Partition} (Spring 2020)
    \item \textit{Dominance Properties of the Gini Index} (Spring 2020)
    \item \textit{Representation Theory of the Dihedral Group} (Fall 2021)
    \item \textit{The Gini Index and Representations of the Symmetric Group} (Fall 2021)
\end{itemize}
\item \textbf{MAA Wisconsin Sectional Meeting, UW Milwaukee} \hfill April 2017\\
\textit{``Speial K: Congruences for the k-Regular Partition Function''}
\end{itemize}
\end{rSection}

\begin{rSection}{Awards, Scholarships and Grants}
\begin{itemize}
\item \textbf{Academic Excellence and Service Award}\hfill May 2014\\An award for service and academic excellence in the mathematics department.
\item \textbf{Harvey C. McKenzie Mathematics Award}\hfill May 2017\\
An award for senior mathematics majors recognizing outstanding academic performance.
\item \textbf{Chancellor's Award} \hfill Fall 2017-Spring 2021
\item \textbf{Research Excellence Award} \hfill Fall 2017, Fall 2019\\
An award in recognition of excellence in mathematical research.
\item \textbf{Ernst Schwandt Teaching Assistant Award} \hfill May 2020\\
Ernst Schwandt Memorial Scholarship and Teaching Assistant Award in recognition outstanding teaching performance.
\item \textbf{Mark Lawrence Teply Award} \hfill May 2020\\
An award in recognition of outstanding research potential.
\end{itemize}
\end{rSection}

\begin{rSection}{Professional Development}
\begin{itemize}
    \item \textbf{UW Milwaukee Teaching Seminar}\hfill{Spring 2021}
    \item \textbf{Active Learning for Equitable Instruction}\hfill{Summer 2020}\\
    Participated in system-wide professional development course.
\end{itemize}
\end{rSection}

\begin{rSection}{Memberships}
\begin{itemize}
    \item American Mathematical Society
\end{itemize}
\end{rSection}

\begin{rSection}{Computing Skills}
\begin{itemize}
    \item Mathematical Software -- Maple, Mathematica
    \item Programming Languages -- VBA, Java, Latex
    \item Online Homework Systems -- WeBWork, Aleks, Wiley Connect, Realizeit
    \item Course Management Systems -- D2L, Canvas
    \item Other -- Microsoft Excel, Power BI, Sharepoint, Dynamics
\end{itemize}
\end{rSection}

\begin{rSection}{References}
\begin{itemize}
    \item Dr. Boris Okun\\ 
    University of Wisconsin, Milwaukee\\
    e-mail: okun@uwm.edu\\
    phone: (414) 251-7188
    \item Dr. Jeb Willenbring (Advisor)\\ 
    University of Wisconsin, Milwaukee\\
    e-mail: jw@uwm.edu\\
    phone: (414) 229-5280
    \item Dr. Kevin McLeod\\
    University of Wisconsin, Milwaukee\\
    e-mail: kevinm@uwm.edu\\
    phone: (414) 229-5269
\end{itemize}
\underline{Mailing Address for all references}:\\ P. O. Box 413\\
Department of Mathematical Sciences\\
University of Wisconsin-Milwaukee\\
Milwaukee, WI 53201-0413\\
USA
\end{rSection}
\newpage
\let\clearpage\relax


\newpage
\addcontentsline{toc}{chapter}{References}
\chapter*{References}
\begin{bibdiv}
\bibliographystyle{alphabetic}
\begin{biblist}

\setlength{\itemsep}{-1mm}

\bib{Andrews}{book}{
   author={Andrews, G. E.},
   author={Eriksson, K.},
   title={Integer partitions},
   publisher={Cambridge University Press, Cambridge},
   date={2004},
   pages={x+141},
   isbn={0-521-84118-6},
   isbn={0-521-60090-1},
   review={\MR{2122332}},
   doi={10.1017/CBO9781139167239},
}

\bib{Inequalities}{book}{
   author={Marshall, A. W.},
   author={Olkin, I.},
   author={Arnold, B. C.},
   title={Inequalities: theory of majorization and its applications},
   series={Springer Series in Statistics},
   edition={2},
   publisher={Springer, New York},
   date={2011},
   pages={xxviii+909},
   isbn={978-0-387-40087-7},
   review={\MR{2759813}},
   doi={10.1007/978-0-387-68276-1},
}

\bib{Brylawski}{article}{
   author={Brylawski, T.},
   title={The lattice of integer partitions},
   journal={Discrete Math.},
   volume={6},
   date={1973},
   pages={201--219},
   issn={0012-365X},
   review={\MR{325405}},
   doi={10.1016/0012-365X(73)90094-0},
}

\bib{Origins}{article}{       author={Ceriani, L.},
    author={Verme, P.}, title={The origins of the Gini index: extracts from variabilit\'a e mutabilit\'a (1912) by Corrado Gini},
    journal={Journal of Economic Inequality}, publisher={Springer}, date={2012},
    pages={421--443},
    doi={10.1007/s10888-011-9188-x},
    number={10},
}


\bib{Chevalley}{article}{
   author={Chevalley, C.},
   title={Invariants of finite groups generated by reflections},
   journal={Amer. J. Math.},
   volume={77},
   date={1955},
   pages={778--782},
   issn={0002-9327},
   review={\MR{72877}},
   doi={10.2307/2372597},
}

\bib{Hall-Littlewood}{article}{
   author={D\'{e}sarm\'{e}nien, J.},
   author={Leclerc, B.},
   author={Thibon, J. Y.},
   title={Hall-Littlewood functions and Kostka-Foulkes polynomials in
   representation theory},
   language={English, with English and French summaries},
   journal={S\'{e}m. Lothar. Combin.},
   volume={32},
   date={1994},
   pages={Art. B32c, approx. 38},
   review={\MR{1399504}},
}

\bib{Early}{article}{
   author={Early, E.},
   title={Chain lengths in the dominance lattice},
   journal={Discrete Math.},
   volume={313},
   date={2013},
   number={20},
   pages={2168--2177},
   issn={0012-365X},
   review={\MR{3084260}},
   doi={10.1016/j.disc.2013.05.016},
}
	

\bib{Farris}{article}{
   author={Farris, F. A.},
   title={The Gini index and measures of inequality},
   journal={Amer. Math. Monthly},
   volume={117},
   date={2010},
   number={10},
   pages={851--864},
   issn={0002-9890},
   review={\MR{2759359}},
   doi={10.4169/000298910X523344},
}
		

\bib{Fulton}{book}{
   author={Fulton, W.},
   title={Young tableaux},
   series={London Mathematical Society Student Texts},
   volume={35},
   note={With applications to representation theory and geometry},
   publisher={Cambridge University Press, Cambridge},
   date={1997},
   pages={x+260},
   isbn={0-521-56144-2},
   isbn={0-521-56724-6},
   review={\MR{1464693}},
}

\bib{SRI}{book}{
   author={Goodman, Roe},
   author={Wallach, Nolan R.},
   title={Symmetry, representations, and invariants},
   series={Graduate Texts in Mathematics},
   volume={255},
   publisher={Springer, Dordrecht},
   date={2009},
   pages={xx+716},
   isbn={978-0-387-79851-6},
   review={\MR{2522486}},
   doi={10.1007/978-0-387-79852-3},
}
		

\bib{Greene}{article}{
   author={Greene, C.},
   author={Kleitman, D. J.},
   title={Longest chains in the lattice of integer partitions ordered by
   majorization},
   journal={European J. Combin.},
   volume={7},
   date={1986},
   number={1},
   pages={1--10},
   issn={0195-6698},
   review={\MR{850140}},
   doi={10.1016/S0195-6698(86)80013-0},
}


	

\bib{Hall}{article}{
    author={Hall, P.},
    title={The algebra of partitions},
    journal={Proceedings of the 4th Canadian Mathematical Conference},
    pages={147--159},
    date={1957},
}	

\bib{Hardy}{article}{
   author={Hardy, G. H.},
   author={Ramanujan, S.},
   title={Asymptotic formul\ae  in combinatory analysis [Proc. London Math.
   Soc. (2) {\bf 17} (1918), 75--115]},
   conference={
      title={Collected papers of Srinivasa Ramanujan},
   },
   book={
      publisher={AMS Chelsea Publ., Providence, RI},
   },
   date={2000},
   pages={276--309},
   review={\MR{2280879}},
}
		
\bib{Hesselink}{article}{
   author={Hesselink, Wim H.},
   title={Characters of the nullcone},
   journal={Math. Ann.},
   volume={252},
   date={1980},
   number={3},
   pages={179--182},
   issn={0025-5831},
   review={\MR{593631}},
   doi={10.1007/BF01420081},
}
		
\bib{Kopitzke}{article}{
   author={Kopitzke, G.},
   title={The Gini index of an integer partition},
   journal={J. Integer Seq.},
   volume={23},
   date={2020},
   number={9},
   pages={Art. 20.9.7, 13},
   review={\MR{4167937}},
}

\bib{Kostant}{article}{
   author={Kostant, Bertram},
   title={Lie group representations on polynomial rings},
   journal={Bull. Amer. Math. Soc.},
   volume={69},
   date={1963},
   pages={518--526},
   issn={0002-9904},
   review={\MR{150240}},
   doi={10.1090/S0002-9904-1963-10980-5},
}	

\bib{Lascoux-Schutzenberger}{article}{
   author={Lascoux, A.},
   author={Sch\"{u}tzenberger, M.},
   title={Sur une conjecture de H. O. Foulkes},
   language={French, with English summary},
   journal={C. R. Acad. Sci. Paris S\'{e}r. A-B},
   volume={286},
   date={1978},
   number={7},
   pages={A323--A324},
   issn={0151-0509},
   review={\MR{472993}},
}
		

\bib{Littlewood}{article}{
   author={Littlewood, D. E.},
   title={On certain symmetric functions},
   journal={Proc. London Math. Soc. (3)},
   volume={11},
   date={1961},
   pages={485--498},
   issn={0024-6115},
   review={\MR{130308}},
   doi={10.1112/plms/s3-11.1.485},
}
	

\bib{Lorenz}{article}{         author={Lorenz, M. O.},      journal = {Publications of the American Statistical Association},
 number = {70},
 pages = {209--219},
 publisher = {American Statistical Association, Taylor & Francis, Ltd.},
 title = {Methods of measuring the concentration of wealth},
 volume = {9},
 date = {1905},
}

\bib{Macdonald}{book}{
   author={Macdonald, I. G.},
   title={Symmetric functions and Hall polynomials},
   series={Oxford Classic Texts in the Physical Sciences},
   edition={2},
   note={With contribution by A. V. Zelevinsky and a foreword by Richard
   Stanley;
   Reprint of the 2008 paperback edition [ MR1354144]},
   publisher={The Clarendon Press, Oxford University Press, New York},
   date={2015},
   pages={xii+475},
   isbn={978-0-19-873912-8},
   review={\MR{3443860}},
}
		

\bib{OEIS}{article}{
 author={OEIS Foundation Inc.}, 
 name={The On-Line Encyclopedia of Integer Sequences},
 date={2021}, url={http://oeis.org},
}

\bib{Shephard-Todd}{article}{
   author={Shephard, G. C.},
   author={Todd, J. A.},
   title={Finite unitary reflection groups},
   journal={Canad. J. Math.},
   volume={6},
   date={1954},
   pages={274--304},
   issn={0008-414X},
   review={\MR{59914}},
   doi={10.4153/cjm-1954-028-3},
}

\bib{Stanley GL}{article}{
   author={Stanley, R. P.},
   title={${\rm GL}(n,\mathbb{ C})$ for combinatorialists},
   conference={
      title={Surveys in combinatorics},
      address={Southampton},
      date={1983},
   },
   book={
      series={London Math. Soc. Lecture Note Ser.},
      volume={82},
      publisher={Cambridge Univ. Press, Cambridge},
   },
   date={1983},
   pages={187--199},
   review={\MR{721186}},
}

\bib{Stanley}{article}{
   author={Stanley, R. P.},
   title={Invariants of finite groups and their applications to
   combinatorics},
   journal={Bull. Amer. Math. Soc. (N.S.)},
   volume={1},
   date={1979},
   number={3},
   pages={475--511},
   issn={0273-0979},
   review={\MR{526968}},
   doi={10.1090/S0273-0979-1979-14597-X},
}


\end{biblist}
\end{bibdiv}
\end{document}